\documentclass[3p]{elsarticle}

\newcommand{\rzn}{\mathbb R^{n}}
\newcommand{\rz}{\mathbb R}
\newcommand{\rzm}{\mathbb R^{m}}
\newcommand{\nat}{\mathbb N_0}
\newcommand{\natu}{\mathbb N}
\usepackage{amsmath,mathrsfs,amssymb,amsfonts,amsthm,newtxmath}

\newtheorem{thm}{Theorem}
\newtheorem{prop}[thm]{Proposition}
\newtheorem{lem}[thm]{Lemma}
\newtheorem{defi}[thm]{Definition}
\newtheorem{assu}[thm]{Assumption}
\newtheorem{cor}[thm]{Corollary}
\newtheorem{examp}[thm]{Example}
\newtheorem{remark}[thm]{Remark}

\journal{JMMA}

\begin{document}
	
	\begin{frontmatter}
		
		
		\title{The maximum principle for discrete-time control systems and applications
			to dynamic games}
		
		\author[rvt]{Alberto Dom\'inguez Corella\corref{cor1}}
		\ead{adominguez@math.cinvestav.mx}
		\author[rvt]{On\'esimo Hern\'andez-Lerma}
		\ead{ohernand@math.cinvestav.mx}

		\address[rvt]{Mathematics Department, CINVESTAV-IPN, A. Postal 14-740, Mexico City 07000, Mexico}
		
		\cortext[cor1]{Corresponding author}


		\begin{abstract}
			We study deterministic nonstationary discrete-time  optimal  control  problems  in both finite and  infinite  horizon. With the aid of Gateaux differentials, we prove a discrete-time maximum principle in analogy with the well-known continuous-time maximum principle. We show that this maximum principle, together with a transversality condition, is a  necessary condition for optimality; we also show that it is sufficient under additional hypotheses.  We use Gateaux differentials as a natural setting to derive first-order conditions. Additionally, we use the discrete-time maximum principle to derive the discrete-time Euler equation and to characterize Nash equilibria for discrete-time dynamic games.
		\end{abstract}
		
		
	\end{frontmatter}
	
	\section{Introduction}
	In this paper, we study first-order conditions for optimality in  deterministic nonstationary discrete-time control systems. Our main purpose is to establish a discrete maximum principle, in analogy with the continuous-time case, and a transversality condition as a necessary and sufficient condition for optimality.
	The book by Blot and Hayek \cite{Blot} gives an account of the state of the art in the infinite horizon framework; however, there is not a single example in that book, which suggests that the results in \cite{Blot} are not easily applicable. In contrast, we illustrate our results with several examples and applications.
	
	The continuous-time maximum principle is well known in the literature; see \cite{Pont}; \cite{Flemming} and \cite{Clarke} for a thorough presentation of the theory; \cite{HistoryMP} for its history. We are interested in an analogous principle for discrete-time systems. In \cite[Section 3.2.3]{Blot}, Blot and Hayek prove a maximum principle for trajectories and state processes that are both bounded. Since their proof relies on the classical optimization techniques of functional analysis \cite{Clarke,Lue}, they obtain the existence of the adjoint variables in a non-constructive way. We work with arbitrary trajectories and give an explicit form of the adjoint variables.
	In \cite{Multipliers}, 	Aseev et al. developed  necessary conditions in the form of a maximum principle for weakly overtaking solutions based on ideas from the continuous-time setting. We work under similar assumptions.
	
	Here, we present a direct approach based on the use of Gateaux differentials as a natural setting to obtain the maximum principle for open-loop strategies. More explicitly, under a technical assumption, we calculate a Gateaux derivative of the performance index and use it  to obtain necessary conditions. In addition, we consider convexity-concavity assumptions to establish sufficient conditions. 
	
	We present several applications and extensions of our maximum principle. Namely, control problems with state constraints. For that purpose, we use Markov strategies and prove that a similar maximum principle holds. We briefly consider finite-horizon optimal control problems to show that it is a particular case of our model.
	We derive the Euler equation together with a new transversality condition as a particular case of our maximum principle. The Euler equation has been recently studied in \cite{EEDavid} for open-loop strategies. In contrast, we propose an approach with Markov strategies. As an extension of our work for optimal control problems, we consider discrete-time dynamic games in infinite horizon and use the maximum principle to characterize open-loop Nash equilibria.
	
	This paper is organized as follows. In Section 2, we introduce the optimal control model we are interested in and obtain the maximum principle and a transversality condition as a first-order optimality condition; we also present a couple of examples to illustrate our results. In Section 3, we consider related topics: Markov strategies, finite horizon and the Euler equation approach. Finally, in Section 4 we obtain a maximum principle for dynamic games.

	\section{The Maximum Principle}\label{MPSec}
	\subsection{The Optimal Control Model}\label{olmodel}
	
	In this subsection, we present the control model we will be dealing with, which concerns discrete-time  nonstationary  (or  time-varying) deterministic dynamic optimization problems in infinite horizon. Dynamic optimization problems are also known as optimal control problems. 
	
	As usual,  $\natu$ denotes the set $\left\lbrace 1,2,\dots \right\rbrace $ and $\nat$ denotes $\natu\cup\left\lbrace 0\right\rbrace .$
	
	Let $X\subset\rzn$ be the state space, and $U\subset\rzm$ the control set. Consider a sequence $\left\lbrace X_t\hspace{0.1cm}|\hspace{0.1cm} t\in\nat\right\rbrace$ of non-empty subsets of $X$, and  $\left\lbrace U_t\subset U\hspace{0.1cm}|\hspace{0.1cm} t\in\nat\right\rbrace $ a family of feasible control sets. For each $t\in\nat$, $x\in X_t$, and $u\in U_t$, we denote by $f_t(x,u)$ the corresponding state in $X_{t+1}$. Thus, given an initial state $x_0$, the state of the system evolves according to the equation
	\begin{align}\label{dynamic}
	x_{t+1}=f_t(x_t,u_t) \hspace{.5cm}\forall_{t\in\nat},
	\end{align}
	where, for each $t\in\nat$, $f_t:X_t\times U_t\to X_{t+1}$ is a given system function. We want to optimize the performance index
	\begin{align}\label{Performanceindex}
	\sum_{t=0}^{\infty}g_t(x_t,u_t),
	\end{align}
	\\
	where $g_t:X_t\times U_t\to\rz$ is a given function for each $t\in\nat$.

	A sequence $\psi=\left\lbrace u_t\right\rbrace $ is called an open-loop strategy, or simply a plan, whenever $u_t$ is in $ U_t$ for all $t\in\nat$. We denote the set of plans from $x_0$ as $\Psi$.
	Given a plan $\psi=\left\lbrace u_t \right\rbrace $, we denote by $\left\lbrace x_t^\psi \right\rbrace $ the sequence induced by $\psi$ in (\ref{dynamic}), i.e., 
	\begin{align*}
	&x_0^\psi=x_0\\
	&x_{t+1}^\psi=f_t(x_t^\psi,u_t) \hspace{0.3cm}\text{for $t=0,1,\dots$}.
	\end{align*}

	The Optimal Control Problem (OCP) we are concerned is to find a plan $\psi$, also called control policy, that maximizes the performance index (\ref{Performanceindex}) subject to (\ref{dynamic}). 
	
	In a compact form, a nonstationary OCP can be described by the triplet
	\begin{align}\label{tuple}
	(\Psi, \left\lbrace f_t\right\rbrace, \left\lbrace g_t\right\rbrace),
	\end{align} 
	with components as above. 
	
	For the OCP to be well-defined, the following assumption is supposed to hold throughout the remainder of this section.

	\begin{assu}\label{WellPosed}
		The triplet in (\ref{tuple}) satisfies the following for each $x_0\in X_0$:
		\begin{enumerate}
			\item[(a)] the set $\Psi$ is nonempty;
			
			\item[(b)] for each $\psi=(u_0,u_1,\dots)\in\Psi$, 
			\begin{align*}
			\sum_{t=0}^\infty g_t(x_t^\psi,u_t)<\infty;
			\end{align*}
			\item[(c)] there exists $\psi=(u_0,u_1,\dots)\in\Psi$ such that 
			\begin{align*}
			\sum_{t=0}^\infty g_t(x_t^\psi,u_t)>-\infty;
			\end{align*}
			\item[(d)] for each $t\in\nat$, $f_t$ and $g_t$ are differentiable in the interior of $X_t\times U_t$.
		\end{enumerate}
	\end{assu}
	For $x_0\in X_0$, define the OCP performance index (also known as objective function) $v:\Psi\to\mathbb R\cup\left\lbrace -\infty\right\rbrace $ by
	\begin{align}\label{PI}
	v(\psi)=\sum_{t=0}^\infty g_t(x_t^\psi,u_t).
	\end{align}
	Assumption \ref{WellPosed}(a)-(b) ensures that $v$ is well defined.
	
	For the triplet (\ref{WellPosed}) and $x_0\in X_0$, the OCP is to find $\hat\psi\in\Psi$ such that 
	\begin{align*}
	v(\hat{\psi})\ge v(\psi)
	\end{align*} 
	for all $\psi\in\Psi$. In such a case, we say that $\hat{\psi}$ is an optimal plan. The optimization problem makes sense by Assumption \ref{WellPosed}(c).

	\subsection{Necessary Conditions}\label{NECCON}

	In this subsection, we introduce the Maximum Principle (MP)  (\ref{MPX})-(\ref{MPY}) and the Transversality Condition (TC) (\ref{TC}) as necessary conditions for the existence of an optimal plan. 
	We suppose that the initial state $x_0\in X_0$ is fixed. Recall that Assumption \ref{WellPosed} holds. We will require the concept of Gateaux differential.
	
	\begin{defi}{\cite[pp.2-4]{Flemming},\cite[p.171]{Lue}}
		Let $\mathcal{X}$ be a linear space and $\mathcal{V}$ a subset of $\mathcal X$. Let $p\in\mathcal V$ and $q\in\mathcal X$.
		\begin{enumerate}
			\item[(a)]  We say that $p$ is an internal of $\mathcal V$  point in the direction $q$ if there exist a real number $\varepsilon_0>0$ such that 
			$
			p+\varepsilon q
			$ is in $\mathcal V$
			for all $\varepsilon\in(-\varepsilon_0,\varepsilon_0)$.
			\item[(b)]Suppose $p$ is an internal point of $\mathcal V$ in the direction $q$. Let $h :\mathcal V\to\rz$ be a function. If the derivative
			\begin{align*}
			\delta_{h}(p;q)&:=\left.\frac{d h}{d\varepsilon}(p+\varepsilon q)\right|_{\varepsilon=0}
			\end{align*}
			exists, we say that $\delta_{h}(p;q)$ is the Gateaux differential of $h$ at $p$ in the direction $q$.
		\end{enumerate}
	\end{defi}
	
	The next proposition shows an application of G\^ ateaux differentials.
	
	\begin{prop}\label{GD}
		Let $\mathcal{X}$ be a linear space and $\mathcal{V}$ a subset of $\mathcal X$. Let $p\in\mathcal V$ be an internal point in the direction $q\in\mathcal X$ and $h:\mathcal V\to\rz$ a given function. If the Gateaux differential of $h$ at $p$ in the  direction $q$ exists and $h$ has a maximum at $p$, then 
		\begin{align*}
		\delta_h(p;q)=0.
		\end{align*}
	\end{prop}
	\begin{proof}
		See Theorem 1 in page 178 of \cite{Lue}.
	\end{proof}

	The plan to prove the MP (\ref{MPX})-(\ref{MPY}) and the TC (\ref{TC}) below is straightforward: we will calculate the Gateaux differential, in a certain direction,  of the performance index (\ref{PI}) at the optimal plan, and then use Proposition \ref{GD}. 
	
	We will need the following assumption to estimate the Gateaux differential of the function $v$ in (\ref{PI}). As usual, $\displaystyle\frac{\partial}{\partial x}$ and $\displaystyle\frac{\partial}{\partial y}$ denote the gradients corresponding to the first and the second variables, respectively.

	\begin{assu}\label{AMP}
		Let  $\hat \psi=(\hat u_0,\hat u_1,\dots) \in \Psi$. For each  $\tau\in\nat$, define the sequence of functions $\rho^\tau_t: U_\tau\to\rzn$ as $$\rho_t^\tau(u):=\frac{\partial g_t}{\partial x}( x_t^{{\hat\psi^\tau(u)}}, \hat u_t)\prod_{s=\tau+1}^{t-1}\frac{\partial f_s}{\partial x}(x_s^{{\hat \psi^\tau(u)}}, \hat u_s),$$ 
		where $\hat\psi^\tau(u)=(\hat u_0,\dots,\hat u_{\tau-1},u,\hat u_{\tau+1},\dots)$. 
		We suppose that, for each $\tau\in\nat$, there exists an open neighborhood $\mathcal O_\tau\subset U_\tau$ of $\hat u_\tau$ such that the series $\displaystyle\sum_{t=\tau+1}^{\infty}\rho^\tau_t$ converges uniformly on $\mathcal O_\tau$. 
	\end{assu}

	If 
	$A_1, A_2\dots$ is a sequence of square matrices, we write
	\begin{align*}
	\prod_{s=\tau}^t A_s=\left\{ \begin{array}{lcc}
	A_t\dots A_{\tau+1}A_{\tau} &   if  & \tau\le t \\
	\\ \mathbb I &  if &  \tau>t ,
	\end{array}
	\right.
	\textit{} \end{align*}
	where $\mathbb I$ is the identity matrix.

	Computing the Gateaux differential of $v$ may be too  technical; calculations are in the following lemma. The proof is in Section \ref{secproo}.
	We consider the vector space $\Lambda$ of all sequences in $\mathbb R^m$ with the standard addition and scalar multiplication. We consider row vectors $y\in\mathbb R^m$, so the transpose $y^*$ is a column vector.
	\begin{lem}\label{DV}
		Let  $\hat \psi=\left\lbrace \hat u_t\right\rbrace \in \Psi$ be a plan for which Assumption \ref{AMP} holds. Let $y\in\rzm$ and $\tau\in\nat$. Then $\hat\psi$ is an internal point of $\Psi$ in the direction $\psi^{\tau,y}\in\Lambda$, where $\psi^{\tau,y}$ is defined as 
		\begin{align*}
		\psi^{\tau,y}_t:=\left\{ \begin{array}{lcc}
		\displaystyle y &   if  & t=\tau \\
		\\ 0 &  if & t\neq\tau, \\
		\end{array}
		\right.
		\end{align*}
		for all $t\in\nat$. Moreover, if $v$ is the function in (\ref{PI}), its G\^ ateaux differential at $\hat\psi$ in the direction $\psi^{\tau,y}$ exists and is given by 
		\begin{align*}
		\delta_v(\hat\psi;\psi^{\tau,y})= \left( \sum_{t=\tau+1}^\infty\frac{\partial g_t}{\partial x}(x_t^{\hat{\psi}},\hat u_t)\prod_{s=\tau+1}^{t-1}\frac{\partial f_s}{\partial x}(x_s^{\hat{\psi}},\hat u_s)\right)\frac{\partial f_\tau}{\partial y}(x_\tau^{\hat{\psi}},\hat u_\tau)y^*+\frac{\partial g_\tau}{\partial y}(x_\tau^{\hat{\psi}},\hat u_\tau)y^*.
		\end{align*} 
	\end{lem} 
	
	We can now state one of our main results.
	\begin{thm}\label{MP}
		Let $\hat \psi=\left\lbrace\hat u_t \right\rbrace \in\Psi$ be a plan for which Assumption \ref{AMP} holds. If $\hat\psi$ is an optimal plan of the OCP (\ref{dynamic})-(\ref{tuple}), then there exists a sequence $\left\lbrace \lambda_t\right\rbrace_{t=1}^\infty$ in $\rzn$ such that
		\begin{enumerate}
			\item[(a)] For all $t\in\natu$,
			\begin{align}\label{MPX}
			\frac{\partial g_t}{\partial x}(x_t^{\hat\psi},\hat u_t)+\lambda_{t+1}\frac{\partial f_t}{\partial x}(x_t^{\hat\psi},\hat u_t)=\lambda_t,
			\end{align}
			\item[(b)] For all $t\in\nat$,
			\begin{align}\label{MPY}
			\frac{\partial g_t}{\partial y}(x_t^{\hat\psi},\hat u_t)+\lambda_{t+1}\frac{\partial f_t}{\partial y}(x_t^{\hat\psi},\hat u_t)=0,
			\end{align}
			\item[(c)] For every $h\in\natu$, we have the  transversality condition (TC)
			\begin{align}\label{TC}
			\lim_{t\to\infty}\lambda_t \prod_{s=h}^{t-1}\frac{\partial f_s}{\partial x}(x_s^{\hat\psi},\hat u_s)=0.
			\end{align}  
		\end{enumerate}
		Moreover, each $\lambda_t$ is given by 
		\begin{align}\label{multipliers}
		\lambda_t=\sum_{k=t}^\infty\frac{\partial g_k}{\partial x}(x_k^{\hat{\psi}},\hat u_k)\prod_{s=t}^{k-1}\frac{\partial f_s}{\partial x}(x_s^{\hat{\psi}},\hat u_s).
		\end{align}
	\end{thm}
	\begin{proof}
		
		\begin{enumerate}
			\item[(a)] Pick an arbitrary $\tau\in\natu$. Then
			\begin{align*}
			\lambda_\tau&:=\sum_{t=\tau}^\infty\frac{\partial g_t}{\partial x}(x_t^{\hat{\psi}},\hat u_t)\prod_{s=\tau}^{t-1}\frac{\partial f_s}{\partial x}(x_s^{\hat{\psi}},\hat u_s)\\
			&=\frac{\partial g_\tau}{\partial x}(x_\tau^{\hat{\psi}},\hat u_\tau)+\sum_{t=\tau+1}^\infty\frac{\partial g_t}{\partial x}(x_t^{\hat{\psi}},\hat u_t)\prod_{s=\tau}^{t-1}\frac{\partial f_s}{\partial x}(x_s^{\hat{\psi}},\hat u_s)\\
			&=\frac{\partial g_\tau}{\partial x}(x_\tau^{\hat{\psi}},\hat u_\tau)+\left( \sum_{t=\tau+1}^\infty\frac{\partial g_t}{\partial x}(x_t^{\hat{\psi}},\hat u_t)\prod_{s=\tau+1}^{t-1}\frac{\partial f_s}{\partial x}(x_s^{\hat{\psi}},\hat u_s)\right)\frac{\partial f_\tau}{\partial x}(x_\tau^{\hat{\psi}},\hat u_\tau)\\
			&=\frac{\partial g_\tau}{\partial x}(x_\tau^{\hat{\psi}},\hat u_\tau)+\lambda_{\tau+1}\frac{\partial f_\tau}{\partial x}(x_\tau^{\hat{\psi}},\hat u_\tau).
			\end{align*}
			\item[(b)] Fix $\tau\in\nat$ and $y\in\rzm$ arbitrary. By Lemma \ref{DV} and Proposition \ref{GD},
			\begin{align*}
			\left( \sum_{t=\tau+1}^\infty\frac{\partial g_t}{\partial x}(x_t^{\hat{\psi}},\hat u_t)\prod_{s=\tau+1}^{t-1}\frac{\partial f_s}{\partial x}(x_s^{\hat{\psi}},\hat u_s)\right)\frac{\partial f_\tau}{\partial y}(x_\tau^{\hat{\psi}},\hat u_\tau)y^*+\frac{\partial g_\tau}{\partial y}(x_\tau^{\hat{\psi}},\hat u_\tau)y^*=0,
			\end{align*}
			that is 
			\begin{align*}
			\left[\frac{\partial g_\tau}{\partial y}(x_\tau^{\hat{\psi}},\hat u_\tau)y^*+\lambda_{\tau+1}\frac{\partial f_\tau}{\partial y}(x_\tau^{\hat{\psi}},\hat u_\tau)y^*\right]=0.
			\end{align*}
			Since this holds for any $y\in\rzm$, $(b)$ follows.
			\item[(c)] By Assumption \ref{AMP}, the series
			\begin{align*}
			\sum_{k=h}^\infty\frac{\partial g_k}{\partial x}(x_k^{\hat{\psi}},\hat u_k)\prod_{s=h}^{k-1}\frac{\partial f_s}{\partial x}(x_s^{\hat{\psi}},\hat u_s),
			\end{align*}
			converges for any $h\in\natu$ and
			\begin{align*}
			\lambda_t \prod_{s=h}^{t-1}\frac{\partial f_s}{\partial x}(x_s^{\hat\psi},\hat u_s)&=\left( \sum_{k=t}^\infty\frac{\partial g_k}{\partial x}(x_k^{\hat{\psi}},\hat u_k)\prod_{s=t}^{k-1}\frac{\partial f_s}{\partial x}(x_s^{\hat{\psi}},\hat u_s)\right)  \prod_{s=h}^{t-1}\frac{\partial f_s}{\partial x}(x_s^{\hat\psi},\hat u_s)\\
			&=\sum_{k=t}^\infty\frac{\partial g_k}{\partial x}(x_k^{\hat{\psi}},\hat u_k)\prod_{s=h}^{k-1}\frac{\partial f_s}{\partial x}(x_s^{\hat{\psi}},\hat u_s).
			\end{align*}
			Letting $t$ tend to infinity, $(c)$ follows.
		\end{enumerate}
	\end{proof}
	
	Actually, the MP (\ref{MPX})-(\ref{MPY}) is so-named
	in analogy with the well-known maximum principle in the continuous-time case (see \cite{Pont}). See \cite{HistoryMP} for the history of the continuous-time maximum principle.
	
	\begin{remark}
		If we allow each $f_t$ and $g_t$ to be continuously differentiable, we can use the lines of proof of Theorem~2.2 in \cite{Multipliers} to obtain Theorem~\ref{MP}. The idea of our proof can be used to prove the case of Markov strategies and control sets that depend on the state, which is one of our main results. The explicitly defined adjoint sequence given by \eqref{multipliers} (in the present paper) as well as the transversality condition \eqref{TC} are taken from \cite{Multipliers}.
	\end{remark}

	It turns out that the sequence $\left\lbrace \lambda_t \right\rbrace $ given in Theorem \ref{MP} must be unique.

	\begin{prop}\label{mul}
		Suppose that a plan $\hat\psi\in \Psi$ satisfies (\ref{MPX}) and the TC (\ref{TC}). Then 
		\begin{align}\label{mul2}
		\lambda_t=\sum_{k=t}^\infty\frac{\partial g_k}{\partial x}(x_k^{\hat{\psi}},\hat u_k)\prod_{s=t}^{k-1}\frac{\partial f_s}{\partial x}(x_s^{\hat{\psi}},\hat u_s).
		\end{align}
	\end{prop}
	\begin{proof}
		Let $\left\lbrace \lambda'_t\right\rbrace_{t=1}^\infty $ be a sequence satisfying (\ref{MPX}) and (\ref{TC}). It can be proved by induction that
		\begin{align*}
		\lambda'_t=\sum_{s=t}^{h}\frac{\partial g_s}{\partial x}(x_s^{\hat{\psi}},\hat u_s)\prod_{i=t}^{s-1}\frac{\partial f_i}{\partial x}(x_i^{\hat{\psi}},\hat u_i)+\lambda'_{h+1}\prod_{i=t}^{h}\frac{\partial f_i}{\partial x}(x_i^{\hat{\psi}},\hat u_i),
		\end{align*}
		for $h\ge t$. Now, letting h tend to infinite 
		\begin{align*}
		\lambda'_t=\sum_{s=t}^{\infty}\frac{\partial g_s}{\partial x}(x_s^{\hat{\psi}},\hat u_s)\prod_{i=t}^{s-1}\frac{\partial f_i}{\partial x}(x_i^{\hat{\psi}},\hat u_i).
		\end{align*}
		Comparing with (\ref{mul2}), the result follows.
	\end{proof}

	\begin{cor}\label{mulunique}
		The sequence $\left\lbrace \lambda_t\right\rbrace_{t=1}^\infty $ given in Theorem \ref{MP} is unique.
	\end{cor}

	\subsection{Sufficient Conditions}\label{SUFCON}
	
	We have seen that the MP (\ref{MPX})-(\ref{MPY}) and the TC (\ref{TC}) are necessary conditions for optimality. Under suitable assumptions, they are also sufficient; see Theorem  \ref{SufThm} and Assumption \ref{SuffCon} below.
	
	As in the previous subsection, we require a proposition concerning Gateaux differentials.

	\begin{prop}\label{SGD}
		Let $\mathcal{X}$ be a linear space.
		Suppose $\mathcal V$ is a convex subset of $\mathcal X$ and $h:\mathcal V\to\rz$ a concave function. If $p\in\mathcal V$ satisfies $\delta_h(p;q-p)=0$ for all $q\in\mathcal V$, then $p$ maximizes $h$.
	\end{prop}
	\begin{proof}
		Since $\mathcal V$ is convex, $p$ is an internal point in the direction $q-p$ for any $q\in\mathcal V$.
		By concavity of $h$,
		\begin{align*}
		h(p+\varepsilon(q-p))\ge h(p)+\varepsilon(h(q)-h(p)),
		\end{align*}
		for all $0\le\varepsilon\le1$. That is
		\begin{align*}
		\frac{h(p+\varepsilon(q-p))-h(p)}{\varepsilon}\ge h(q)-h(p).
		\end{align*}
		Letting $\varepsilon\downarrow0$ yields the result.
	\end{proof}

	The following convexity-concavity assumption ensures the sufficiency of the MP (\ref{MPX})-(\ref{MPY}) and the TC (\ref{TC}). 
	
	\begin{assu}
		The optimal control model (\ref{tuple}) satisfies the following:
		\begin{enumerate}\label{SuffCon}
			\item[(a)]  the set of plans $\Psi$ is convex;	\item[(b)]	the performance index $v$ in (\ref{PI}) is concave;
			\item[(c)]	 there exists a sequence of real  numbers $\left\lbrace m_t\right\rbrace $ such that $\sum_{t=0}^\infty m_t$ converges and 
			$
			g_t(x_t^\psi,u_t)\ge m_t
			$ 
			for all $\psi=(u_0,u_1,\dots)\in\Psi$.
		\end{enumerate}
	\end{assu}

	\begin{thm}\label{SufThm}
		Let $\hat{\psi}=\left\lbrace \hat u_t\right\rbrace\in\Psi$ be a plan for which Assumption \ref{AMP} holds. Suppose that $\hat\psi$ satisfies the MP (\ref{MPX})-(\ref{MPY}) and the TC (\ref{TC}). If Assumption \ref{SuffCon} holds, then $\hat\psi$ is an optimal plan for the OCP (\ref{dynamic})-(\ref{tuple}).
	\end{thm}
	
	\begin{proof}
		For each $k\in\nat$, consider the function $v_{\hat{\psi}}^k:\Psi\to\rz$ given by $$v_{\hat\psi}^k(u_0,u_1,\dots)=v(u_0,\dots,u_{k},\hat u_{k+1},\hat u_{k+2},\dots),$$ where $v$ is the performance index (\ref{PI}).
		Proceeding as in the proof of Lemma \ref{DV}, we find
		\begin{align*}
		&\delta_{v_{\hat{\psi}}^k}(\hat\psi,\psi-\hat \psi)=\sum_{\tau=0}^k\left(\sum_{t=\tau+1}^\infty\left[ \frac{\partial g_t}{\partial x}(x_t^{\hat\psi},\hat u_t)\prod_{s=\tau+1}^{t-1}\frac{\partial f_s}{\partial x}(x_s^{\hat\psi},\hat u_s)\right]\frac{\partial f_\tau}{\partial y}(x_\tau^{\hat\psi},\hat u_\tau) +\frac{\partial g_\tau}{\partial y}(x_\tau^{\hat\psi},\hat u_\tau)\right)(u_\tau-\hat u_\tau)^*.
		\end{align*}
		By Proposition \ref{mul} and (\ref{MPY}), we have $\delta_{v_{\hat{\psi}}^k}(\hat\psi,\psi)=0$, which by Proposition \ref{SGD} yields that $\hat{\psi}$ is a maximum of $v^k_{\hat{\psi}}$. Let $\psi\in\Psi$ be any plan. By Assumption \ref{SuffCon}(c)
		\begin{align*}
		v(\hat{\psi})&=v_{\hat{\psi}}^k(\hat \psi)\\
		&\ge v_{\hat{\psi}}^k(\psi)\\
		&\ge \sum_{t=0}^k g_t(x_t^\psi,u_t) + \sum_{t=k+1}^\infty m_t.
		\end{align*}
		Letting $k\to\infty$, we obtain $v(\hat \psi)\ge v(\psi).$
	\end{proof}
	
	\subsection{Examples} 
	\begin{examp}[A consumption-investment problem]
		Let $\beta,\gamma\in(0,1)$ and $r>0$ such that  $\displaystyle (r\beta)^{\frac{1}{\gamma}}<r$. Assume that $x_t$ is the wealth of certain investor at time $t\in\nat$. At each time $t=0,1,\dots$, the investor consumes a fraction $u_t\in(0,1)$ of the assets. Suppose that the investor wishes to maximize the following discounted utility of consumption
		
		\begin{align*}
		\sum_{t=0}^\infty \beta^t(x_tu_t)^{1-\gamma},
		\end{align*}\
		
		subject to the dynamics of the assets
		
		$$x_{t+1}=r(1-u_t)x_t,$$ where $x_0>0$ is given.
		
		In the present context, our control model in Section \ref{olmodel} has the following components
		
		\begin{itemize}
			\item state space $X_t\equiv X:=(0,\infty)$;
			\item control space $U_t\equiv U:=(0,1)$;
			\item system functions $f_t:X\times U\to X$ with $f_t(x,u):=r(1-u)x$;
			\item return functions $g_t:X\times U\to \rz$ with $g_t(x,u):=\beta^t (xu)^{1-\gamma}$.
		\end{itemize}

		To use Theorem \ref{MP}, we proceed as follows. From MP (\ref{MPX})-(\ref{MPY}):

		\begin{align}
		\lambda_t=\beta^t(1-\gamma)(x_t^{\hat\psi}\hat u_t)^{-\gamma}\hat u_t+\lambda_{t+1}r(1-\hat u_t) \hspace{0.3cm}\forall_{t\in\natu},
		\end{align}
		
		\begin{align}
		0=\beta^t(1-\gamma)(x_t^{\hat \psi} \hat u_t)^{-\gamma}-\lambda_{t+1}r \hspace{0.3cm}\forall_{t\in\nat}.
		\end{align} 
		
		Combining these equations we obtain
		$\lambda_{t}=\beta^t(1-\gamma)(x_t^{\hat\psi}\hat u_t)^{-\gamma}$ and 
		\begin{align}\label{lambdasol}
		1/r= \frac{\lambda_{t+1}}{\lambda_t}=\beta\left( \frac{x_{t+1}^{\hat\psi} \hat u_{t+1}}{x_t^{\hat\psi} \hat u_t} \right)^{-\gamma}.
		\end{align}
	Using the fact that $x_{t+1}^{\hat\psi}=r(1-\hat u_t)x_t^{\hat\psi}$, it follows that
	\begin{align*}
	(r\beta)^{\frac{1}{\gamma}}=\frac{x_{t}^{\hat\psi}\hat u_{t}}{x_{t-1}^{\hat\psi}\hat u_{t-1}}=\frac{rx_t^{\hat\psi}-x_{t+1}^{\hat\psi}}{rx_{t-1}^{\hat\psi}-x_t^{\hat\psi}},
	\end{align*}
	which can be written as the second-order difference equation
	\begin{align}\label{eqicp}
	x_{t+1}^{\hat\psi}-[(r\beta)^{\frac{1}{\gamma}}+r]x_t^{\hat\psi}+r(r\beta)^{\frac{1}{\gamma}}=0.
	\end{align}
	The general solution of (\ref{eqicp}) is 
	\begin{align*}
	x_t^{\hat\psi}=c_1(r\beta)^{\frac{t}{\gamma}}+c_2r^t,
	\end{align*}
	for some constants $c_1$ and $c_2$.
	
	Observe that 
	\begin{align*}
	\lambda_t\prod_{s=1}^{t-1}r(1-\hat u_s)&=\lambda_t\prod_{s=1}^{t-1}r(1-\hat u_s)\frac{x_s^{\hat\psi}}{x_s^{\hat\psi}}\\
	&=\lambda_t\prod_{s=1}^{t-1}\frac{x_{s+1}^{\hat\psi}}{x_s^{\hat\psi}}\\
	&=\lambda_t\frac{x_t^{\hat\psi}}{x_1^{\hat\psi}}.
	\end{align*}
	Thus, the TC (\ref{TC}) can be seen as
	\begin{align}\label{tcicp}
	\lambda_tx_t^{\hat\psi}\to0.
	\end{align}
	From (\ref{lambdasol}), we have $\lambda_t=\lambda_1r^{1-t}$. Now
	
	\begin{align*}
	\lambda_tx_t^{\hat\psi}=\lambda_1c_1(r\beta)^{\frac{1}{\gamma}}\left[ \frac{(r\beta)^{\frac{1}{\gamma}}}{r}\right]^{t-1}+\lambda_1c_2r.
	\end{align*}
	Since  $\displaystyle (r\beta)^{\frac{1}{\gamma}}<r$, by (\ref{tcicp}) and the initial condition $x_0$, we conclude $c_2=0$ and $c_1=x_0$; thus $x_t^{\hat\psi}=x_0(r\beta)^{\frac{1}{\gamma}}$. Therefore $\hat u_t=1-\displaystyle\frac{(r\beta)^{\frac{1}{\gamma}}}{r}$ for all $t\in\nat$.
		
		We prove that Assumption \ref{AMP} holds. Let $\tau\in\nat$ and consider $\rho_t^\tau:[0,1]\to\rz$ as in the assumption. Observe that 
		\begin{align*}
		x_t^\psi&=r(1-u_{t-1})x^\psi_{t-1}\\
		&=r(1-u_{t-1})r(1-u_{t-2})x^\psi_{t-2}\\
		& \hspace{0.1cm}\vdots\\
		&=\prod_{s=\tau+1}^{t-1}r(1-u_s)x^\psi_{\tau+1}\\
		&=\prod_{s=\tau+1}^{t-1}\frac{\partial f_s}{\partial x}(x_s^\psi,u_s)x^\psi_{\tau+1},
		\end{align*}
		for any plan $\psi$. Since $u_t\in(0,1)$ for all $t\in\nat$, we have $|x_t^\psi|<r^t|x_0|$ for any plan $\psi$.
		
		Now, take $\mathcal O_\tau=(\eta',\eta)$ as a small neighborhood of $1-\displaystyle\frac{(r\beta)^{\frac{1}{\gamma}}}{r}$ properly contained in $(0,1)$, we have
		\begin{align*}
		\left|\rho_t^{\tau}(u)\right|&=\left|\frac{\partial g_t}{\partial x}( x_t^{{\hat\psi^\tau(u)}}, \hat u_t)\prod_{s=\tau+1}^{t-1}\frac{\partial f_s}{\partial x}(x_s^{{\hat \psi^\tau(u)}}, \hat u_s)\right|\\
		&=\left|\beta^tr(1-\gamma)(x_t^{\hat\psi(u)}\hat u_t)^{-\gamma}\hat u_t\frac{x_t^{\hat\psi(u)}}{x_{\tau+1}^{\hat\psi(u)}}\right|\\
		&=\left|\beta^t(1-\gamma)(x_t^{\hat\psi(u)}\hat u_t)^{1-\gamma}\frac{1}{(1-u)x_{\tau}^{\hat\psi}}\right|\\
		&<\frac{\left|x
			_0\right|^{1-\gamma}}{(1-\eta) x_{\tau}^{\hat\psi}}(\beta r^{1-\gamma})^t.
		\end{align*}
		Since $\displaystyle\frac{(r\beta)^{\frac{1}{\gamma}}}{r}<1$ implies $\beta r^{1-\gamma}<1$, we have by the Weierstrass M-test that  $\displaystyle\sum_{t=\tau+1}^{\infty}\rho_t^\tau$ converges uniformly on $\mathcal O_\tau$.
	\end{examp}

	\begin{examp}[A linear regulator problem]\label{LQ}
		An OCP with linear system equation and a quadratic cost function is known as a \textit{LQ problem} (also called a  \textit{linear regulator problem}). LQ problems have been widely studied. See, for instance, Chapter 5 of \cite{Ljungq}.
		We consider a particular deterministic scalar case. The state of the system evolves according to 
		\begin{align}\label{lqdy}
		x_{t+1}=x_t+u_t,
		\end{align}
		for $t\in\nat$. The performance index is
		\begin{align}\label{lqpi}
		\sum_{t=0}^\infty\beta^t\left[ \frac{1}{2}x_t^2+\frac{1}{2}u_t^2\right],
		\end{align}
		where $0<\beta<1$.
		Given $x_0\in\mathbb R$, we want to minimize (\ref{lqpi}) subject to (\ref{lqdy}).

		In the present context, our control model in Section \ref{olmodel} has the following components
		
		\begin{itemize}
			\item state space $X_t\equiv X:=\mathbb R$;
			\item control space $U_t\equiv U:=\mathbb R$;
			\item system functions $f_t:X\times U\to X$ with $f_t(x,u):=x+u$;
			\item cost functions $g_t:X\times U\to \rz$ with $g_t(x,u):=\frac{\beta^t}{2}\left[x^2+u^2 \right] $.
		\end{itemize}

		Considering (\ref{MPX})-(\ref{MPY}) of Theorem \ref{MP}:

		\begin{align}\label{eq1icp}
		\lambda_t=\beta^tx_t^{\hat{\psi}}+\lambda_{t+1} \hspace{0.3cm} \forall_{t\in\natu},
		\end{align}
		\begin{align}
		0=\beta^t \hat u_t+\lambda_{t+1} \hspace{0.3cm} \forall_{t\in\nat}.
		\end{align} 
		
		From these equations, we obtain 
		\begin{align*}
		\beta^tx^{\hat{\psi}}_t&=\lambda_t-\lambda_{t+1}\\
		&=-\beta^{t-1}\hat u_{t-1}+\beta^t\hat u_t\\
		&=-\beta^{t-1}(x_t^{\hat{\psi}}-x^{\hat{\psi}}_{t-1})+\beta^t(x^{\hat{\psi}}_{t+1}-x^{\hat{\psi}}_{t}),
		\end{align*}
		which is equivalent to the difference equation
		\begin{align}\label{LQEQ}
		\beta x^{\hat{\psi}}_{t+1}-(1+2\beta)x^{\hat{\psi}}_t+x^{\hat{\psi}}_{t-1}=0.
		\end{align}
		The solution of (\ref{LQEQ}) is $x^{\hat{\psi}}_t=k_1r_1^t+k_2r_2^t$ for some constants $k_1$ and $k_2$; where $r_1$ and $r_2$ are the roots of equation $\beta x^2-(1+2\beta)x+1=0$.
		The transversality condition $(\ref{TC})$ reduces to
		\begin{align*}
		\lambda_t\to0.
		\end{align*}
		From this fact and  $(\ref{eq1icp})$, we conclude that $\beta^t x_t^{\hat{\psi}}\to 0$. Now, 
		\begin{align*}
		\beta^tx^{\hat{\psi}}_t&=k_1(\beta r_1)^t+k_1(\beta r_2)^t.
		\end{align*}
		Since $\beta r_1>1 $ and $\beta r_2<1$, we conclude that  $k_1=0$, and by the initial condition, $k_2=x_0$.
		So $$\displaystyle\hat u_t=x_0r_2^{t+1}-x_0r_2^t=(r_2-1)r_2^tx_0.$$

		To prove that Assumption \ref{AMP} holds, let $\tau\in\nat$ and consider $\rho_t^\tau:\rz\to\rz$ as in the assumption. Take $\mathcal O_\tau=(-x_0,x_0)$ , then we have
		\begin{align*}
		\left|\rho_{t}^{\tau}(u)\right|&=\left|\frac{\partial g_t}{\partial x}( x_t^{{\hat\psi^\tau(u)}}, \hat u_t)\prod_{s=\tau+1}^{t-1}\frac{\partial f_s}{\partial x}(x_s^{{\hat \psi^\tau(u)}}, \hat u_s)\right|\\
		&=\left|\beta^tx_t^{\hat\psi^\tau(u)}\right|\\
		&=\beta^t\left|x_\tau^{\hat\psi}+u+\sum_{s=\tau+1}^{t-1}(r_2-1)x_0r_2^t\right|\\
		&<\left|x_0\right||t-\tau+1|\beta^t.
		\end{align*}
		Thus, by the Weierstrass M-test, $\displaystyle\sum_{t=\tau+1}^{\infty}\rho^\tau_t$ converges uniformly on $\mathcal O_\tau$.
	\end{examp}
	
	\subsection{Proof of Lemma \ref{DV}}\label{secproo}
	
	For the reader's conveniencek, we restate here Lemma \ref{DV}.
	\\\\
	\textbf{Lemma \ref{DV}.}
	\textit{Let  $\hat \psi=\left\lbrace \hat u_t\right\rbrace \in \Psi$ be a plan for which Assumption \ref{AMP} holds. Let $y\in\rzm$ and $\tau\in\nat$. Then $\hat\psi$ is an internal point of $\Psi$ in the direction $\psi^{\tau,y}\in\Lambda$, where $\psi^{\tau,y}$ is defined as 
		\begin{align*}
		\psi^{\tau,y}_t:=\left\{ \begin{array}{lcc}
		\displaystyle y &   if  & t=\tau \\
		\\ 0 &  if & t\neq\tau, \\
		\end{array}
		\right.
		\end{align*}
		for all $t\in\nat$. Moreover, if $v$ is the function in (\ref{PI}), its G\^ ateaux differential at $\hat\psi$ in the direction $\psi^{\tau,y}$ exists and is given by 
		\begin{align*}
		\delta_v(\hat\psi;\psi^{\tau,y})= \left( \sum_{t=\tau+1}^\infty\frac{\partial g_t}{\partial x}(x_t^{\hat{\psi}},\hat u_t)\prod_{s=\tau+1}^{t-1}\frac{\partial f_s}{\partial x}(x_s^{\hat{\psi}},\hat u_s)\right)\frac{\partial f_\tau}{\partial y}(x_\tau^{\hat{\psi}},\hat u_\tau)y^*+\frac{\partial g_\tau}{\partial y}(x_\tau^{\hat{\psi}},\hat u_\tau)y^*.
		\end{align*} }
	
	\begin{proof}
		First, we prove that $\hat\psi$ is an internal point in the direction $\psi^{\tau,y}$.
		If $t=\tau$; there exists $\varepsilon_\tau>0$ such that $\hat u_\tau+\varepsilon y\in U_\tau$ for all $\varepsilon\in(-\varepsilon_\tau,\varepsilon_\tau)$, since, by assumption, $\hat u_\tau$ belongs to an open neighborhood $\mathcal O_\tau\subset U_\tau$.  So  $\hat\psi+\varepsilon\psi^{\tau,y}$ is in $\Psi$ for all $\varepsilon\in(-\varepsilon_\tau,\varepsilon_\tau)$. To prove the assertion about the Gateaux differential, we compute the derivatives of the functions in (a)-(c) below.
		\begin{enumerate}
			\item[(a)] For each $t\ge\tau$, define $h_t:(-\varepsilon_\tau,\varepsilon_\tau)\to \rzn$ as $h_t(\varepsilon):=f_t(x_t^{\hat{\psi}+\varepsilon\psi^{\tau,y}},\hat u_t+\varepsilon\psi_t^{\tau,y})$.
			We prove by induction that
			\begin{align*}
			h'_t(\varepsilon)=\displaystyle\left[ \prod_{s=\tau+1}^t\frac{\partial f_s}{\partial x}(x_s^{\hat{\psi}+\varepsilon\psi^{\tau,y}},\hat u_s)\right] \frac{\partial f_\tau}{\partial y}(x_\tau^{\hat{\psi}+\varepsilon\psi^{\tau,y}},\hat u_\tau+\varepsilon y)y^*.
			\end{align*}
			If $t=\tau$,
			\begin{align*}
			h'_{\tau}(\varepsilon)=\frac{\partial f_{\tau}}{\partial y}(x_t^{\hat\psi},\hat u_{\tau}+\varepsilon y)y^*.
			\end{align*}
			Suppose that it holds for $t$. Then, by the chain rule,
			\begin{align*}
			h'_{t+1}(\varepsilon)&=[ f_{t+1}(x_{t+1}^{\hat{\psi}+\varepsilon\psi^{\tau,y}},\hat u_{t+1}+\varepsilon\psi^{\tau,y}_{t+1})]'\\
			&=\left[ f_{t+1}(h_t(\varepsilon),\hat u_{t+1}+\varepsilon\psi^{\tau,y}_{t+1})\right]'\\
			&=\frac{\partial f_{t+1}}{\partial x}(h_t(\varepsilon),\hat u_{t+1}+\varepsilon\psi^{\tau,y}_{t+1})h_t'(\varepsilon)+\frac{\partial f_{t+1}}{\partial y}(h_t(\varepsilon),\hat u_{t+1}+\varepsilon\psi^{\tau,y}_{t+1})[\psi^{\tau,y}_{t+1}]^*\\
			&=\displaystyle\left[ \prod_{s=\tau+1}^{t+1}\frac{\partial f_s}{\partial x}(x_s^{\hat{\psi}+\varepsilon\psi^{\tau,y}},\hat u_s+\varepsilon\psi^{\tau,y}_s)\right] \frac{\partial f_{\tau}}{\partial y}(x_\tau^{\hat{\psi}+\varepsilon\psi^{\tau,y}},\hat u_\tau+\varepsilon y)y^*.
			\end{align*}
			\item[(b)] Define, for $t>\tau$, $k_t:(-\varepsilon_\tau,\varepsilon_\tau)\to\rz$ as
			\begin{align*}
			k_t(\varepsilon):=g_t(x_t^{\hat{\psi}+\varepsilon\psi^{\tau,y}},\hat u_t+\varepsilon\psi^{\tau,y}_t).
			\end{align*}
			For $t>\tau$, we have
			\begin{align*}
			k_t'(\varepsilon)&=[g_t(x_t^{\hat{\psi}+\varepsilon\psi^{\tau,y}},\hat u_t+\varepsilon\psi^{\tau,y}_t)]'\\
			&=[g_t(h_{t-1}(\varepsilon),\hat u_t+\varepsilon\psi^{\tau,y}_t)]'\\
			&=\frac{\partial g_t}{\partial x}(x_t^{\hat{\psi}+\varepsilon\psi^{\tau,y}},\hat u_t+\varepsilon\psi^{\tau,y}_t)h_{t-1}'(\varepsilon)+\frac{\partial g_t}{\partial y}(x_t^{\hat{\psi}+\varepsilon\psi^{\tau,y}},\hat u_t+\varepsilon\psi^{\tau,y}_t)[\psi^{\tau,y}_{t}]^*\\
			&=\frac{\partial g_t}{\partial x}(x_t^{\hat{\psi}+\varepsilon\psi^{\tau,y}},\hat u_t)\displaystyle\left[ \prod_{s=\tau+1}^{t-1}\frac{\partial f_s}{\partial x}(x_s^{\hat{\psi}+\varepsilon\psi^{\tau,y}},\hat u_s)\right] \frac{\partial f_\tau}{\partial y}(x_\tau^{\hat{\psi}+\varepsilon\psi^{\tau,y}},\hat u_\tau+\varepsilon y)y^*
			\end{align*}
			\item[(c)] For each $T\in\nat$, define $l_T:\mathcal (-\varepsilon_\tau,\varepsilon_\tau)\to\mathbb R$ as 
			\begin{align*}
			l_T(\varepsilon):=\sum_{t=0}^T g_t(x_t^{\hat{\psi}+\varepsilon\psi^{\tau,y}},\hat u_t+\varepsilon \psi_t^{\tau,y})
			\end{align*}
			For $T>\tau$, we have 
			\begin{align*}
			l'_T(\varepsilon)&=\left[ \sum_{t=0}^T g_t(x_t^{\hat{\psi}+\varepsilon\psi^{\tau,y}},\hat u_t+\varepsilon \psi_t^{\tau,y}) \right]'\\
			&=\left[ \sum_{t=0}^{\tau-1} g_t(x_t^{\hat{\psi}},\hat u_t) \right]'+[ g_\tau(x_\tau^{\hat{\psi}},\hat u_\tau+\varepsilon y)]'+\left[ \sum_{t=\tau+1}^T g_t(x_t^{\hat{\psi}+\varepsilon\psi^{\tau,y}},\hat u_t) \right]'\\
			&= \frac{\partial g_\tau}{\partial y}(x_\tau^{\hat{\psi}+\varepsilon\psi^{\tau,y}},\hat u_\tau+\varepsilon y)y^*+ \sum_{t=\tau+1}^T k_t'(\varepsilon)\\
			&= \sum_{t=\tau+1}^T\left[\frac{\partial g_t}{\partial x}(x_t^{\hat{\psi}+\varepsilon\psi^{\tau,y}},\hat u_t)\displaystyle\left( \prod_{s=\tau+1}^{t-1}\frac{\partial f_s}{\partial x}(x_s^{\hat{\psi}+\varepsilon\psi^{\tau,y}},\hat u_s)\right) \frac{\partial f_\tau}{\partial y}(x_\tau^{\hat{\psi}+\varepsilon\psi^{\tau,y}},\hat u_\tau+\varepsilon y)y^*\right]\\
			&\hspace{0.5cm}+ \frac{\partial g_\tau}{\partial y}(x_\tau^{\hat{\psi}+\varepsilon\psi^{\tau,y}},\hat u_\tau+\varepsilon y)y^*.
			\end{align*}
			
		\end{enumerate}
		Finally, from (a), (b) and (c), we obtain
		\begin{align*}
		\delta_v(\hat{\psi};\psi^{\tau,y})&=\left.\frac{d}{d\varepsilon}\left[ v(\hat\psi+\varepsilon\psi^{\tau,y})\right]\right|_{\varepsilon=0}\\
		&=\frac{d}{d\varepsilon}\left.\left[ \lim_{T\to\infty}l_T(\varepsilon)\right]\right|_{\varepsilon=0}\\
		&=\left.\lim_{T\to\infty}\left[ l'_T(\varepsilon)\right]\right|_{\varepsilon=0}\\
		&= \sum_{t=\tau+1}^\infty\left[\frac{\partial g_t}{\partial x}(x_t^{\hat{\psi}},\hat u_t)\displaystyle\left( \prod_{s=\tau+1}^{t-1}\frac{\partial f_s}{\partial x}(x_s^{\hat{\psi}},\hat u_s)\right) \frac{\partial f_\tau}{\partial y}(x_\tau^{\hat{\psi}},\hat u_\tau)y^*\right]+ \frac{\partial g_\tau}{\partial y}(x_\tau^{\hat{\psi}},\hat u_\tau)y^*\\
		&= \left( \sum_{t=\tau+1}^\infty\frac{\partial g_t}{\partial x}(x_t^{\hat{\psi}},\hat u_t)\prod_{s=\tau+1}^{t-1}\frac{\partial f_s}{\partial x}(x_s^{\hat{\psi}},\hat u_s)\right)\frac{\partial f_\tau}{\partial y}(x_\tau^{\hat{\psi}},\hat u_\tau)y^*+\frac{\partial g_\tau}{\partial y}(x_\tau^{\hat{\psi}},\hat u_\tau)y^* .
		\end{align*}
		Assumption \ref{AMP}  ensures the interchange between the limit and the derivative.
	\end{proof}

	\section{Variants of The Maximum Principle}

	\subsection{Finite Horizon Optimal Control Problems}

	In this subsection we consider again the non-stationary OCP (\ref{dynamic})-(\ref{tuple}), except that the performance index (\ref{Performanceindex}) is now replaced by the finite-horizon function
	\begin{align}\label{Performanceindexfh}
	\sum_{t=0}^{T-1}g_t(x_t,u_t)+g_T(x_T).
	\end{align}
	In particular, the dynamic control model is as in (\ref{dynamic}), that is
	\begin{align}\label{dynamicfh}
	x_{t+1}=f_t(x_t,u_t),
	\end{align}
	for $t\in\left\lbrace 0,\dots, T-1\right\rbrace$, with a given initial condition $x_0$.

	As before, given a plan $\psi=(u_0,\dots,u_{T-1}) $, we denote by $\left\lbrace x_t^\psi \right\rbrace $ the sequence induced by $\psi$ in (\ref{dynamicfh}), i.e., 
	\begin{align}
	&x_0^\psi=x_0\\
	&x_{t+1}^\psi=f_t(x_t^\psi,u_t).
	\end{align}

	In this optimal control model, we want to find a plan $\psi$, also called a control policy, that maximizes the performance index (\ref{Performanceindexfh}) subject to (\ref{dynamicfh}).
	
	In compact form, the optimal control model can be described by the triplet.
	\begin{align}\label{tuplefh}
	(\Psi_T(x_0), \left\lbrace f_t\right\rbrace, \left\lbrace g_t\right\rbrace),
	\end{align} 
	with components as above. 
	
	The following assumption is supposed to hold throughout the remainder of this subsection.

	\begin{assu}\label{WellPosedfh}
		The triplet in (\ref{tuplefh}) satisfies the following for each $x_0\in X_0$:
		\begin{enumerate}
			\item[(a)] the set $\Psi_T(x_0)$ is nonempty;

			\item[(b)] for each $t\in\left\lbrace 0,\dots, T-1\right\rbrace $, $f_t$ and $g_t$ are differentiable in the interior of $X_t\times U_t$ and $g_T$ in the interior of $X_T$.
		\end{enumerate}
	\end{assu}
	Throughout the following we fix the initial state $x_0$. Define $v_T:\Psi_T(x_0)\to\mathbb R$ by
	\begin{align}\label{PIfh}
	v_T(\psi)=\sum_{t=0}^{T-1} g_t(x_t^\psi,u_t) + g_T(x_T^\psi).
	\end{align}
	Given (\ref{tuplefh}) and $x_0\in X_0$, The OCP is to find $\hat\psi\in\Psi_T(x_0)$ such that 
	\begin{align*}
	v_T(\hat{\psi})\ge v_T(\psi),
	\end{align*} 
	for all $\psi\in\Psi_T(x_0)$. In this case, we say that $\hat{\psi}$ is an optimal plan.

	The following theorem is a consequence of Theorem \ref{MP} and Corollary \ref{mulunique}.  Observe that the TC (\ref{TC}) reduces to the terminal condition (\ref{TCfh}).
	
	\begin{thm}\label{MPfh}
		Let $\hat \psi=(\hat u_0,\dots, \hat u_{T-1})\in\Psi_T(x_0)$ be a plan such that each $\hat u_t$ is in the interior of $U_t$. If $\hat\psi$ is an optimal plan of the control model (\ref{Performanceindexfh})-(\ref{tuplefh}), then there exist unique $\lambda_1,\dots,\lambda_T$ in $\rzn$ such that
		\begin{enumerate}
			\item[(a)] For all $t\in\left\lbrace 1,\dots, T-1 \right\rbrace $,
			\begin{align}\label{MPXfh}
			\frac{\partial g_t}{\partial x}(x_t^{\hat\psi},\hat u_t)+\lambda_{t+1}\frac{\partial f_t}{\partial x}(x_t^{\hat\psi},\hat u_t)=\lambda_t,
			\end{align}
			\item[(b)] For all $t\in\left\lbrace 0,\dots, T-1 \right\rbrace$,
			\begin{align}\label{MPYfh}
			\frac{\partial g_t}{\partial y}(x_t^{\hat\psi},\hat u_t)+\lambda_{t+1}\frac{\partial f_t}{\partial y}(x_t^{\hat\psi},\hat u_t)=0,
			\end{align}
			\item[(c)] 
			\begin{align}\label{TCfh}
			\lambda_T=\frac{\partial g_T}{\partial x}(x_T^{\hat{\psi}}).
			\end{align}  
		\end{enumerate}
		Moreover, each $\lambda_t$ is given by 
		\begin{align}\label{multipliersfh}
		\lambda_t=\sum_{k=t}^T\frac{\partial g_k}{\partial x}(x_k^{\hat{\psi}},\hat u_k)\prod_{s=t}^{k-1}\frac{\partial f_s}{\partial x}(x_s^{\hat{\psi}},\hat u_s).
		\end{align}
		
	\end{thm}
	
	\begin{proof}
		It suffices to consider the special case of Theorem \ref{MP} in which $g_T$ only depends of the first variable and $g_t\equiv0$ for all $t\ge T+1$, and then apply Corollary \ref{mulunique}.
	\end{proof}

	To establish sufficient conditions, we need the following assumption.
	\begin{assu}
		The control model (\ref{tuplefh}) satisfies the following:
		\begin{enumerate}\label{SuffConfh}
			\item[(a)]  the set of plans $\Psi_T(x_0)$ is convex;	\item[(b)]	the performance index $v_T$ (\ref{PIfh}) is concave.
		\end{enumerate}
	\end{assu}

	The next theorem is consequence of Theorem \ref{SufThm}.

	\begin{thm}\label{SufThmfh}
		Suppose that a plan $\hat{\psi}\in\Psi_T(x_0)$ satisfies (\ref{MPXfh})-(\ref{TCfh}). If Assumption \ref{SuffConfh} holds, then $\hat\psi$ is an optimal plan for the control model (\ref{dynamicfh})-(\ref{tuplefh}).
	\end{thm}
	
	\begin{proof}
		Considering the case when $g_T$ only depends of the first variable and $g_t\equiv0$ for all $t\ge T+1$, Theorem \ref{SufThm} yields the result.
	\end{proof}

	\subsection{Markov Strategies}\label{Markovmodel}

	In this subsection, we present a similar model to (\ref{dynamic})-(\ref{tuple}), but we consider that the control set and the policies may depend of the state.
	
	As usual, let $X\subset\rzn$ be the state space and $U\subset\rzm$ be the control set. Consider a sequence $\left\lbrace X_t\hspace{0.1cm}|\hspace{0.1cm} t\in\nat\right\rbrace$ of nonempty subsets of $X$, and  $\left\lbrace U_t(x)\hspace{0.1cm}|\hspace{0.1cm} x\in X_t,t\in\nat\right\rbrace $ the family of feasible control sets. For each $t\in\nat$, we define $$\mathbb K_t=\left\lbrace (x,u)\hspace{0.1cm}|\hspace{0.1cm} x\in X_t, u\in U_t(x) \right\rbrace. $$ For each $t\in\nat$, $x\in X_t$, and $u\in U_t(x)$. We denote by $f_t(x,u)$ the corresponding state in $X_{t+1}$,  where, for each $t\in\nat$, $f_t:\mathbb K_t\to X_{t+1}$ is a given function.
	
	A sequence $\varphi=\left\lbrace \varphi_t\right\rbrace $ of functions $\varphi_t:X_t\to \rzm$ is called a Markovian strategy whenever $\varphi_t(x)\in U_t(x)$ for all $x\in X_t$, $t\in\nat$. We denote the set of Markovian strategies from $x_0$ as $\varPhi(x_0)$.
	Given a Markovian strategy $\varphi=\left\lbrace \varphi_t \right\rbrace $, we denote by $\left\lbrace x_t^\varphi \right\rbrace $ the state sequence induced by $\varphi$, i.e., 
	\begin{align}\label{dynamicms}
	&x_0^\varphi=x_0\\
	&x_{t+1}^\varphi=f_t(x_t^\varphi,\varphi_t(x_t^\varphi))\hspace{0.3cm}\forall_{t\in\nat}.\label{dynamicms2}
	\end{align}

	We want to optimize 
	\begin{align}\label{Performanceindexms}
	\sum_{t=0}^{\infty}g_t(x_t^\varphi,\varphi_t(x_t^\varphi)),
	\end{align}
	where $g_t:\mathbb K_t\to\rz$ for each $t\in\nat$.  That is, we want to find a Markovian strategy $\varphi\in\varPhi(x_0)$ that maximizes the performance index (\ref{Performanceindexms}). 
	
	In reduced form, the optimal control model can be described by the triplet.
	\begin{align}\label{tuplems}
	(\varPhi(x_0), \left\lbrace f_t\right\rbrace, \left\lbrace g_t\right\rbrace),
	\end{align} 
	with components as above. 
	
	For the OCP to be well defined, the following assumption is supposed to hold throughout the remainder of the subsection.

	\begin{assu}\label{WellPosedms}
		The triplet in (\ref{tuplems}) satisfies the following for each $x_0\in X_0$:
		\begin{enumerate}
			\item[(a)] the set $\varPhi(x_0)$ is nonempty;
			
			\item[(b)] for each $\varphi=(\varphi_0,\varphi_1,\dots)\in\varPhi(x_0)$, 
			\begin{align*}
			\sum_{t=0}^\infty g_t(x_t^\varphi,\varphi_t(x_t^\varphi))<\infty;
			\end{align*}
			\item[(c)] there exists  $\varphi=(\varphi_0,\varphi_1,\dots)\in\varPhi(x_0)$ such that 
			\begin{align*}
			\sum_{t=0}^\infty g_t(x_t^\varphi,\varphi_t(x_t^\varphi))>-\infty;
			\end{align*}
			\item[(d)] for each $t\in\nat$, $f_t$ and $g_t$ are differentiable in the interior of $\mathbb K_t$.
		\end{enumerate}
	\end{assu}
	For $x_0\in X_0$, define the performance index $v:\varPhi(x_0)\to\mathbb R\cup\left\lbrace -\infty\right\rbrace $ by
	\begin{align}\label{PIms}
	v(\varphi)=\sum_{t=0}^\infty g_t(x_t^\varphi,\varphi_t(x_t^\varphi)).
	\end{align}
	Assumption \ref{WellPosedms}(a)-(b) ensures that the function $v$  is well defined. For the triplet (\ref{tuplems}) and $x_0\in X_0$, the OCP is to find $\hat\varphi\in\varPhi(x_0)$ such that 
	\begin{align*}
	v(\hat{\varphi})\ge v(\varphi),
	\end{align*} 
	for all $\varphi\in\varPhi(x_0)$. If this holds, we say that $\hat{\varphi}$ is an optimal plan. The optimization problem makes sense by  Assumption \ref{WellPosedms}(c).

	\begin{remark}
		For notational convenience, for every $t\in\nat$ and $\varphi\in\varPhi(x_0)$, we will write
		\begin{align}\label{notationalconvinience}
		g_t(x,\varphi_t):=g_t(x,\varphi_t(x))\hspace{0.3cm}\text{and}\hspace{0.3cm} f_t(x,\varphi_t):=f_t(x,\varphi_t(x)).
		\end{align}
	\end{remark}

	To proceed as in Section \ref{MPSec}, we need the following assumption, which is analogous to Assumption \ref{AMP}.

	\begin{assu}\label{AMPms}
		Let  $\hat \varphi=(\hat \varphi_0,\hat \varphi_1,\dots) \in \varPhi(x_0)$ be such that each $\hat\varphi_t$ is differentiable in the interior of $X_t$. For each  $\tau\in\nat$, define the sequence of functions $\rho^\tau_t: U_\tau(x_\tau^{\hat\varphi})\to\rzn$ as \begin{align*}
		&\rho^\tau_t(u)=\left[ \frac{\partial g_t}{\partial x}\left( x_t^{{\hat\varphi^\tau(u)}}, \hat \varphi_t\right)+\frac{\partial g_t}{\partial y}\left( x_t^{{\hat\varphi^\tau(u)}}, \hat \varphi_t\right)\frac{\partial \varphi_t}{\partial x}\left( x_t^{\hat{\varphi}^\tau(u)}\right)\right] \prod_{s=\tau+1}^{t-1}\left( \frac{\partial f_s}{\partial x}\left( x_s^{{\hat\varphi^\tau(u)}}, \hat \varphi_s\right) + \frac{\partial f_s}{\partial y}\left( x_s^{{\hat\varphi^\tau(u)}}, \hat\varphi_s\right)\frac{\partial \hat\varphi_s}{\partial x}\left( x_s^{\hat{\varphi}^\tau(u)}\right)\right) ,
		\end{align*}
		where $\hat\varphi^\tau(u)=(\hat \varphi_0,\dots,\hat \varphi_{\tau-1},\varphi_u,\hat \varphi_{\tau+1},\dots)$ and $\varphi_u(x)=u$ for all $x\in X_\tau$. 
		Given $\tau\in\nat$, we suppose that there exists an open neighborhood $\mathcal O_\tau\subset U_\tau(x_\tau^{\hat\varphi})$ of $\hat \varphi_\tau(x_\tau^
		{\hat\varphi})$ such that $\displaystyle\sum_{t=\tau+1}^\infty\rho^\tau_{t}$ converges uniformly on $\mathcal O_\tau$. 
		
	\end{assu}
	The following lemma contains the computation of the Gateaux differential of the performance index $v$. Its proof is analogous to the one of Lemma \ref{DV}, except that now we consider the vector space $\Lambda:=\left\lbrace \left\lbrace \varphi_t\right\rbrace \hspace{0.1cm}|\hspace{0.1cm}\varphi_t:X_t\to\rzm, t\in\nat\right\rbrace $.
	\begin{lem}\label{DVms}
		Let  $\hat \varphi=\left\lbrace \hat \varphi_t\right\rbrace \in \varPhi(x_0)$ such that  Assumption \ref{AMPms} holds. Let $y\in\rzm$ and $\tau\in\nat$. Then $\hat\varphi$ is an internal point in the direction $\varphi^{\tau,y}\in\Lambda$, where $\varphi^{\tau,y}$ is defined as
		\begin{align*}
		\varphi^{\tau,y}_t(x):=\left\{ \begin{array}{lcc}
		\displaystyle y & \text{if}\hspace{0.2cm}  t=\tau\hspace{0.2cm} \text{and} \hspace{0.2cm}x=x_\tau^{\hat\varphi} \\
		\\ 0 &\text{otherwise}. \\
		\end{array}
		\right.
		\end{align*}
		Moreover, the G\^ ateaux differential of $v$ at $\hat\varphi$ in the direction $\varphi^{\tau,y}$ exists and is given by 
		\begin{align*}
		\delta_v(\hat\varphi;\varphi^{\tau,y})=\frac{\partial g_\tau}{\partial y}(x_\tau^{\hat{\varphi}},\hat\varphi_\tau)y^*+ \lambda_{\tau+1}\frac{\partial f_\tau}{\partial y}(x_\tau^{\hat{\varphi}},\hat\varphi_\tau)y^*,
		\end{align*} 
		where $v$ is the function in (\ref{PIms}) and 
		\begin{align*}
		&\lambda_{\tau+1}:=\sum_{k=\tau+1}^\infty\left[ \frac{\partial g_k}{\partial x}\left( x_k^{\hat\varphi},\hat \varphi_k\right)+\frac{\partial g_k}{\partial y}\left( x_k^{\hat\varphi},\hat \varphi_k\right)\frac{\partial \hat\varphi_k}{\partial x}\left(x_k^{\hat\varphi}\right)\right]\prod_{s=\tau+1}^{k-1}\left(  \frac{\partial f_s}{\partial x}\left( x_s^{\hat\varphi},\hat \varphi_s\right)+\frac{\partial f_s}{\partial y}\left( x_s^{\hat\varphi},\hat \varphi_s\right)\frac{\partial \hat\varphi_s}{\partial x}\left(x_s^{\hat\varphi}\right)\right)
		\end{align*}
	\end{lem} 
	Repeating the same arguments in Theorem \ref{MP} and with the aid of Lemma \ref{DVms}, we can prove the next theorem.
	\begin{thm}\label{MPms}
		Let $\hat \varphi=\left\lbrace\hat\varphi_t \right\rbrace \in\varPhi(x_0)$ be such that Assumption \ref{AMPms} holds. If $\hat\varphi$ is an optimal plan for the control model (\ref{dynamicms})-(\ref{tuplems}), then there exists  sequence $\left\lbrace \lambda_t\right\rbrace_{t=1}^\infty$ in $\rzn$ such that
		\begin{enumerate}
			\item[(a)] For all $t\in\natu$,
			\begin{align}\label{MPXms}
			\frac{\partial g_t}{\partial x}(x_t^{\hat\varphi},\hat\varphi_t)+\lambda_{t+1}\frac{\partial f_t}{\partial x}(x_t^{\hat\varphi},\hat\varphi_t)=\lambda_t,
			\end{align}
			\item[(b)] For all $t\in\nat$,
			\begin{align}\label{MPYms}
			\frac{\partial g_t}{\partial y}(x_t^{\hat\varphi},\hat\varphi_t)+\lambda_{t+1}\frac{\partial f_t}{\partial y}(x_t^{\hat\varphi},\hat\varphi_t)=0,
			\end{align}
			\item[(c)] For all $h\in\natu$,
			\begin{align}\label{TCms}
			\lim_{t\to\infty}\lambda_t \prod_{s=h}^{t-1}\left[ \frac{\partial f_s}{\partial x}\left( x_s^{\hat\varphi},\hat \varphi_s\right)+\frac{\partial f_s}{\partial y}\left( x_s^{\hat\varphi},\hat \varphi_s\right)\frac{\partial \hat\varphi_s}{\partial x}(x_s^{\hat\varphi})\right]=0.
			\end{align}  
		\end{enumerate}
		Moreover,

		\begin{equation} \label{mulmsthm}
		\begin{split}
		&\lambda_{t}:=\sum_{k=t}^\infty\left[ \frac{\partial g_k}{\partial x}\left( x_k^{\hat\varphi},\hat \varphi_k\right)\prod_{s=t}^{k-1}\left(  \frac{\partial f_s}{\partial x}\left( x_s^{\hat\varphi},\hat \varphi_s\right)+\frac{\partial f_s}{\partial y}\left( x_s^{\hat\varphi},\hat \varphi_s\right)\frac{\partial \hat\varphi_s}{\partial x}(x_s^{\hat\varphi})\right) +\right.\\
		&\hspace{1.7cm}\left.+\frac{\partial g_k}{\partial y}\left( x_k^{\hat\varphi},\hat \varphi_k\right)\frac{\partial \hat\varphi_k}{\partial x}(x_k^{\hat\varphi})\prod_{s=t}^{k-1}\left(  \frac{\partial f_s}{\partial x}\left( x_s^{\hat\varphi},\hat \varphi_s\right)+\frac{\partial f_s}{\partial y}\left( x_s^{\hat\varphi},\hat \varphi_s\right)\frac{\partial \hat\varphi_s}{\partial x}(x_s^{\hat\varphi})\right) \right].
		\end{split}
		\end{equation}

	\end{thm}
	As in proposition \ref{mul} and Corollary \ref{mulunique}, we now obtain the following.
	\begin{prop}\label{mulms}
		Suppose that a plan $\hat\psi\in \Psi$ satisfies the MP (\ref{MPXms})-(\ref{MPYms}) and the TC (\ref{TCms}). Then $\left\lbrace \lambda_t\right\rbrace $ is given by (\ref{mulmsthm}). 
	\end{prop}

	\begin{cor}\label{muluniquems}
		The sequence $\left\lbrace \lambda_t\right\rbrace $ given in Theorem \ref{MPms} is unique.
	\end{cor}

	We can proceed as in Section \ref{SUFCON} to obtain sufficient conditions.
	\begin{assu}
		The control model (\ref{dynamicms})-(\ref{tuplems}) satisfies the following:
		\begin{enumerate}\label{SuffConfhms}
			\item[(a)]  the set of plans $\varPhi(x_0)$ is convex;	\item[(b)]	the performance index $v$ in  (\ref{PIms}) is concave.
			\item[(c)] there exists a sequence of real  numbers $\left\lbrace m_t\right\rbrace $ such that $\sum_{t=0}^\infty m_t$ converges and 
			$
			g_t(x_t^\varphi,\varphi_t)\ge m_t
			$ for all $\varphi=(\varphi_0,\varphi_1,\dots)\in\varPhi(x_0)$.
		\end{enumerate}
	\end{assu}

	\begin{thm}\label{SufThmfhms}
		Let $\hat{\varphi}\in\varPhi(x_0)$ such that Assumption \ref{AMPms} holds. Suppose that $\hat\varphi$ satisfies (\ref{MPXms})-(\ref{TCms}). If Assumption \ref{SuffConfhms} holds, then $\hat\varphi$ is an optimal plan for the control model (\ref{dynamicms})-(\ref{tuplems}).
	\end{thm}

	\begin{examp}[Optimal economic growth \cite{EconomicGrowth,Ljungq}]
		One of the most studied models in economic growth is the Brock and Mirman model. Capital is represented by $x_t$, and $u_t$ denotes the consumption. The system's dynamics is given by
		\begin{align*}
		x_{t+1}=A_tx_t^\alpha-u_t,
		\end{align*}
		where $\alpha\in(0,1)$. The performance index to be maximized is 
		\begin{align*}
		\sum_{t=0}^\infty\beta^t\log u_t.
		\end{align*}
		
		In the present context, our control model (\ref{dynamicms})-(\ref{Performanceindexms}) has the following components:
		
		\begin{itemize}
			\item state space $X_t\equiv X:=(0,\infty)$;
			\item control space $U:=(0,\infty)$ and control constraint sets $U_t(x):=(0,A_tx^\alpha)$ for all $x\in X$;
			\item system functions $f_t:\mathbb K_t\to X$ with $f_t(x,u):=A_tx^\alpha-u$;
			\item cost functions $g_t:\mathbb K_t\to \rz$ with $g_t(x,u)=\beta^t\log u $.
		\end{itemize}
		
		 For an optimal Markov strategy, Theorem \ref{MPms} yields

		\begin{align}\label{abmm}
		0=\beta^t \frac{1}{\hat\varphi_t(x_t^{\hat\varphi})}-\lambda_{t+1}\hspace{0.3cm}\forall_{t\in\nat},
		\end{align} 
		\begin{align}\label{bbmm}
		\lambda_t=\alpha A_t[x_t^{\hat\varphi}]^{\alpha-1}\lambda_{t+1}\hspace{0.3cm}\forall_{t\in\nat}.
		\end{align}

		In \cite[pp.33]{LagrangeMethod}, Chow solved these  equations using the ``guess and verify method"; he proposes a solution of the form ${\hat\varphi}_t(x)=dA_tx^\alpha$. Using this conjecture for $\hat\varphi_t(x_t^{\hat\varphi})$ and combining $(\ref{abmm})$ and $(\ref{bbmm})$, one obtains $\displaystyle\lambda_t=\frac{\alpha\beta^t}{dx_t^{\hat\varphi}}$. We can use this to evaluate $$\displaystyle\lambda_{t+1}=\frac{\alpha\beta^t}{d([A_tx_t^{\varphi}]^\alpha-dA_t[x_t^{\varphi}]^\alpha)},$$ on the right hand side of $(\ref{bbmm})$. Equating coefficients on both sides of $(\ref{bbmm})$, one obtains $d=1-\alpha\beta$. 
		
		We next verify Assumption \ref{AMPms}. Let $\tau\in\nat$ and consider $\rho_t^\tau:\to\rz$ as in the assumption. Observe that 
		\begin{align*}
		\alpha\frac{x_{s+1}^{\hat\varphi^\tau(u)}}{x_s^{\hat\varphi^\tau(u)}}=\frac{\partial f_s}{\partial x}\left( x_s^{{\hat\varphi^\tau(u)}}, \hat \varphi_s\right) + \frac{\partial f_s}{\partial y}\left( x_s^{{\hat\varphi^\tau(u)}}, \hat \varphi_s\right)\frac{\partial \varphi_s}{\partial x}\left( x_s^{\hat{\varphi}^\tau(u)}\right).
		\end{align*}
		Now, take $\mathcal O_\tau=(\eta',\eta)$ as a small neighborhood of $(1-\alpha\beta)A_t x_\tau^{\hat\varphi}$ properly contained in $[0,A_t[x_\tau^{\hat\varphi}]^\alpha]$. Then we have
		\begin{align*}
		\left|\rho_{\tau}(u)\right|&=\left|\frac{\beta^t}{x_t^{\hat\varphi^\tau(u)}}\prod_{s=\tau+1}^{t-1}\alpha\frac{x_{s+1}^{\hat\varphi^\tau(u)}}{x_s^{\hat\varphi^\tau(u)}}\right|\\
		&=\left| \frac{\beta^t \alpha^{t-\tau+1}}{A_\tau[x_\tau^{\hat\varphi}]^\alpha-u}\right| \\
		&<\left|\frac{\beta^t}{A_\tau [x_\tau^{\hat\varphi}]^\alpha-\eta}\right|.
		\end{align*}
		Thus, by the Weierstrass M-test, $\displaystyle\sum_{t=\tau+1}^{\infty}\rho_t^\tau$ converges uniformly on $\mathcal O_\tau$.
	\end{examp}

	\subsection{The Euler Equation}
	
	Let us now go back to the optimal control model (\ref{dynamicms})-(\ref{tuplems}) in Section \ref{Markovmodel}. As usual in the Euler equation approach, we will consider the particular case in which the functions $f_t$ in (\ref{dynamicms})-(\ref{dynamicms2}) satisfy, for each $t\in\nat$,  $f_t(x,u)=u$ for  all $(x,u)\in\mathbb K_t$. We assume this and the well-posedness Assumption \ref{WellPosedms} to hold during this subsection.
	
	The particular form of the dynamic functions means that at time $t$, we are directly determining the following state of the system, since, a plan $\varphi\in\varPhi(x_0)$ will determine $x_{t+1}^\varphi$ identically  as a function $\varphi_t$ of $x_t^\varphi$. Thus, we want to maximize the performance index
	\begin{align}\label{eulerpi}
	v(\varphi)=\sum_{t=0}^\infty g_t(x_t^\varphi,x_{t+1}^{\varphi}),
	\end{align}
	where
	
	\begin{align}\label{eulerdyn}
	&x_0^\varphi\hspace{0.2cm}=x_0\\
	&x_{t+1}^\varphi=\varphi_t(x_t^\varphi)\hspace{0.3cm}\forall_{t\in\nat}\label{eulerdyn2},
	\end{align}
	and for each $t\in\nat$, $\varphi_t(x)\in U_t(x)$ for all $x\in X_t$.

	For this problem, Assumption \ref{AMPms} reduces to the following one.
	\begin{assu}\label{AEE}
		Let  $\hat \varphi=(\hat \varphi_0,\hat \varphi_1,\dots) \in \varPhi(x_0)$ be such that each $\hat\varphi_t$ is differentiable in the interior of $X_t$. For each  $\tau\in\nat$, define the sequence of functions $\rho^\tau_t: U_\tau(x_\tau^{\hat\varphi})\to\rzn$ as \begin{align*}
		\rho^\tau_t(u)=&\left[  \frac{\partial g_t}{\partial x}\left( x_t^{{\hat\varphi^\tau(u)}}, x_{t+1}^{\hat{\varphi}^\tau(u)}\right) +\frac{\partial g_t}{\partial y}\left( x_t^{{\hat\varphi^\tau(u)}}, x_{t+1}^{\hat{\varphi}^\tau(u)}\right)\frac{\partial \varphi_t}{\partial x}\left( x_t^{\hat{\varphi}^\tau(u)}\right)\right]\prod_{s=\tau+1}^{t-1}\frac{\partial \varphi_s}{\partial x}\left( x_s^{\hat{\varphi}^\tau(u)}\right) ,
		\end{align*}
		where $\hat\varphi^\tau(u)=(\hat \varphi_0,\dots,\hat \varphi_{\tau-1},\varphi_u,\hat \varphi_{\tau+1},\dots)$ and $\varphi_u(x)=u$ for all $x\in X_\tau$. 
		Given $\tau\in\nat$, we suppose that there exists an open neighborhood $\mathcal O_\tau\subset U_\tau(x_\tau^{\hat\varphi})$ of $x_{\tau+1}^
		{\hat\varphi}$ such that $\displaystyle\sum_{t=\tau+1}^\infty\rho^\tau_{t}$ converges uniformly on $\mathcal O_\tau$. 
		
	\end{assu}

	The following theorem is a consequence of Theorem \ref{MPms}. Equation (\ref{MPEE}) is the so-called Euler Equation (EE).

	\begin{thm} Let $\hat\varphi$ be an optimal plan for the control model (\ref{eulerpi})-(\ref{eulerdyn2}). Suppose that $\hat\varphi$ satisfies Assumption \ref{AEE}. Then

		\begin{enumerate}
			\item[(a)] For each $t\in\natu$,

			\begin{align}\label{MPEE}
			\displaystyle\frac{\partial g_{t-1}}{\partial y}\left( x_{t-1}^{\hat{\varphi}},x_{t}^{\hat\varphi}\right) +\displaystyle\frac{\partial g_{t}}{\partial x}\left( x_t^{\hat{\varphi}},x_{t+1}^{\hat\varphi}\right)=0;
			\end{align}
			
			\item[(b)]for each $h\in\natu$,
			\begin{align}\label{TCEE}
			\lim_{t\to\infty}\frac{\partial g_{t-1}}{\partial y}\left( x_{t-1}^{\hat{\varphi}},x_{t}^{\hat\varphi}\right)\prod_{s=h}^{t-1}\frac{\partial \hat\varphi_s}{\partial x}(x_s^{\hat\varphi})=0.
			\end{align}
		\end{enumerate}
		
	\end{thm}
	
	\begin{proof}
		
		From Theorem \ref{MPms} (a)-(b), there exists a sequence $\left\lbrace\lambda \right\rbrace_{t=1}^\infty $ such that, for each $t\in\natu$,
		\begin{align*}
		\lambda_{t}=\frac{\partial g_t}{\partial x}(x_t^{\hat\varphi},x_{t+1}^{\hat\varphi}),
		\end{align*}
		and, for each $t\in\nat$
		\begin{align*}
		0=\frac{\partial g_t}{\partial y}(x_t^{\hat\varphi},x_{t+1}^{\hat\varphi})+\lambda_{t+1}.
		\end{align*}
		These facts yield $(a)$. 
		Part $(b)$, follows from part $(c)$ of Theorem \ref{MPms} and the fact that $\displaystyle\lambda_t=-\frac{\partial g_{t-1}}{\partial y}( x_{t-1}^{\hat{\varphi}},x_{t}^{\hat\varphi})$.
	\end{proof}
	
	To establish sufficient conditions we use the next assumption, which is essentially Assumption \ref{SuffConfhms}.
	
	\begin{assu}
		The control model (\ref{eulerpi})-(\ref{eulerdyn2}) satisfies the following:
		\begin{enumerate}\label{SuffConEE}
			\item[(a)]  the set of plans $\varPhi(x_0)$ is convex;	\item[(b)] the performance index is concave;
			\item[(c)]	there exists a sequence of real  numbers $\left\lbrace m_t\right\rbrace $ such that $\sum_{t=0}^\infty m_t$ converges and 
			$
			g_t(x_t^\varphi,x_{t+1}^\varphi)\ge m_t
			$ 
			for all  $\varphi=(\varphi_0,\varphi_1,\dots)\in\varPhi(x_0)$ .
		\end{enumerate}
	\end{assu}

	Sufficient conditions follow from Theorem \ref{SufThmfhms}.
	
	\begin{thm}\label{SufThmfhEE}
		Let $\hat{\varphi}\in\varPhi(x_0)$ be such that Assumption \ref{AEE} holds. Suppose that $\hat\varphi$ satisfies (\ref{MPEE})-(\ref{TCEE}). If Assumption \ref{SuffConEE} holds, then $\hat\varphi$ is an optimal plan for the control model (\ref{eulerpi})-(\ref{eulerdyn2}).
	\end{thm}
	\begin{proof}
		Define, for each $t\in\natu$,  $\displaystyle\lambda_t:=\frac{\partial g_t}{\partial x}(x_t^{\hat{^\varphi}},x_{t+1}^{\hat{^\varphi}})$. From (\ref{MPEE}), $\displaystyle\lambda_{t+1} =-\frac{\partial g_t}{\partial y}(x_t^{\hat\varphi},x_{t+1}^{\hat\varphi})$ for $t\in\nat$. Thus,
		for each $t\in\natu$,
		\begin{align*}
		\lambda_{t}=\frac{\partial g_t}{\partial x}(x_t^{\hat\varphi},x_{t+1}^{\hat\varphi}).
		\end{align*}
		In addition, for each $t\in\nat$
		\begin{align*}
		0=\frac{\partial g_t}{\partial y}(x_t^{\hat\varphi},x_{t+1}^{\hat\varphi})+\lambda_{t+1}.
		\end{align*}
		Therefore, Theorem \ref{SufThmfhms} yields the result.
	\end{proof}

	\begin{examp}[An economic growth model]
		Consider the following problem concerning an optimal growth model known as de \textit{Ak model}; see Section 2.3.2  of \cite{DavidEE}. Let $\beta\in(0,1)$, $\theta<0$ and $a>1$ such that $(a\beta)^{\frac{1}{\theta-1}}>1$. The performance index is 
		\begin{align*}
		\sum_{t=0}^\infty \frac{\beta^t}{\theta}(ax_{t}-x_{t+1})^{\theta},
		\end{align*}
		subject to $x_{t+1}\in[0,ax_{t}]$, for all $t\in\natu$.
		
		Our control model in this subsection has the following components:
		
		\begin{itemize}
			\item state space $X_t\equiv X:=(0,\infty)$ with control constraint sets $U_t(x)=[0,ax]$ for all $x\in X$;
			\item return functions $g_t:X_{t}\times X_{t+1}\to \rz$ with $g_t(x,u)=\frac{\beta^t}{\theta} (ax-u)^{\theta}$. 
		\end{itemize} Hence, the Euler equation
		\begin{align*}
		-(ax_{t-1}^{\hat\varphi}-x_t^{\hat\varphi})^{\theta-1}+\beta a(ax_{t}^{\hat\varphi}-x_{t+1}^{\hat\varphi})^{\theta-1}=0, \hspace{0.7cm} t=1,2,\dots,
		\end{align*}
		can be expressed as the difference equation
		\begin{align}\label{EEexamp}
		bx_{t+1}^{\hat\varphi}-(1+ab)x_t^{\hat \varphi}+ax_{t-1}^{\hat\varphi}=0,
		\end{align}
		with $b:=(a\beta)^{\frac{1}{\theta-1}}>1$. Considering a linear solution $\hat\varphi_t(x)=\alpha x$, and substituting  $x_{t-1}^{\hat\varphi}$ by $\alpha^{-1}x_t^{\hat\varphi}$ in (\ref{EEexamp}); we obtain $\alpha=b^{-1}$.
		To prove that Assumption \ref{AEE} holds, let $\tau\in\nat$ and consider $\rho_t^\tau:\rz\to\rz$ as in the assumption. Take $\mathcal O_\tau=(\eta,\eta')$ as a small neighborhood of $b^{-1}x_\tau^{\hat\varphi}$ properly contained in $[0,ax_\tau^{\hat\varphi}]$. Then we have
		\begin{align*}
		\left|\rho_{t}^{\tau}(u)\right|&=\left|(a-b^{-1})\beta^t[(a-b^{-1})x_t^{\hat\varphi^\tau(u)}]^{\theta-1}b^{-t+\tau+1}\right|\\
		&<\left|(a-b^{-1})^{\theta}\beta^t [x_t^{\hat\varphi^\tau(u)}]^{\theta-1}\right|\\
		&=\left|(a-b^{-1})^{\theta}\beta^t [b^{\tau-t}u]^{\theta-1}\right|\\
		&<\left|(a-b^{-1})^{\theta}\beta^t\eta^{\theta-1}\right|.
		\end{align*}
		Thus, by the Weierstrass M-test, $\displaystyle\sum_{t=\tau+1}^{\infty}\rho^\tau_t$ converges uniformly on $\mathcal O_\tau$.
	\end{examp}

	\section{Dynamic Games}
	
	In this section, we consider noncooperative dynamic games with $N$ players and state space $X\subset\rzn$.

	Assume that the state dynamics is given by 
	\begin{align}\label{DGdynamic}
	x_{t+1}=f_t(x_t,u_t^{1},\dots,u_t^N),
	\end{align}
	where, for each $j=1,\dots, N$, $u_t^j$  is chosen by player $j$ in the control set $U_t^j\subset\mathbb R^{m_j}$. We suppose that player $j$ wants to ``maximize" a performance index (also known as a payoff function) of the form
	\begin{align}\label{DGPI}
	\sum_{t=0}^{\infty}g_t^j(x_t,u_t^1,\dots,u_t^N),
	\end{align}
	subject to (\ref{DGdynamic}) and a given initial state $x_0$.
	
	We denote by $\Psi^j(x_0)$ the set of plans, or strategies, of player $j$, that is,  $\psi^j=(u_0^j,u_1^j,\dots)$ with $u_t^j\in U_t^j$ for all $t\in\nat$. The set of so-called multistrategies $\psi=(\psi^1,\dots,\psi^N)$ is denoted by $\Psi:=\Psi^1(x_0)\times\dots\times\Psi^N(x_0)$ .

	Given a multistrategy  $\psi=(\psi^1,\dots,\psi^N)\in\Psi$, we denote by $\left\lbrace x_t^\psi \right\rbrace $ the sequence induced by $\psi$ in (\ref{DGdynamic}), i.e., 
	\begin{align*}
	&x_0^\psi=x_0\\
	&x_{t+1}^\psi=f_t(x_t^\psi,u_t^{1},\dots,u_t^N).
	\end{align*}

	We can specify a dynamic game in reduced form as 
	\begin{align}\label{DGtuple}
	\left(\Psi,\left\lbrace f_t\right\rbrace ,\left\lbrace g_t^j\hspace{0.1cm}|\hspace{0.1cm} j=1,\dots,N\right\rbrace\right),
	\end{align} 
	with components as above. 
	
	The following assumption is supposed to hold throughout the remainder of the section.

	\begin{assu}\label{DGWellPosed}
		The triplet in (\ref{DGtuple}) satisfies the following for each $x_0\in X_0$ and each $j=1,\dots,N$:
		\begin{enumerate}
			\item[(a)] the set $\Psi^j(x_0)$ is nonempty,
			
			\item[(b)] for each $\psi\in\Psi$, 
			\begin{align*}
			\sum_{t=0}^{\infty}g_t^j(x_t,u_t^1,\dots,u_t^N)<\infty;
			\end{align*}
			\item[(c)] there exist a $\psi\in\Psi$ such that 
			\begin{align*}
			\sum_{t=0}^\infty g_t^j(x_t^\psi,u_t^1,\dots,u_t^N)>-\infty;
			\end{align*}
			\item[(d)] for each $t\in\nat$, $f_t$ and $g^j_t$ are differentiable in the interior of $X\times U_t^1\times\dots\times U_t^N$.
		\end{enumerate}
	\end{assu}
	For $x_0\in X_0$ and $j=1,\dots N$, define $v^j:\Psi\to\rz$ by
	\begin{align}\label{DGPIv}
	v^j(\psi)=\sum_{t=0}^\infty g_t^j(x_t^\psi,u_t^1,\dots,u_t^N).
	\end{align}
	Assumption \ref{DGWellPosed}(a)-(b) ensures that the function $v^j$ is well defined.
	
	We say that $\hat{\psi}=(\hat\psi^1,\dots,\hat\psi^N)\in\Psi$ is a Nash equilibrium if, for each player $j=1,\dots,N$, 
	\begin{align*}
	v^j(\hat{\psi})\ge v^j(\hat\psi^1,\dots,\hat\psi^{j-1},\psi^j,\hat\psi^{j+1},\dots,\hat\psi^N)\hspace{0.3cm}\forall_{\psi^j\in\Psi^j(x_0)}.
	\end{align*}

	We want to use Theorem \ref{MP} to characterize Nash equilibria (NE). To that end we consider the following  assumption.

	\begin{assu}\label{AMPDG}
		Consider $\hat{\psi}=(\hat\psi^1,\dots,\hat\psi^N)\in\Psi^1(x_0)\times\dots\times\Psi^N(x_0)$. For each  $\tau\in\nat$ and $j=1,\dots, N$, define the sequence of functions $\rho^{\tau,j}_t: U_\tau\to\rzn$ as $$\rho^{\tau,j}_t(u)=\frac{\partial g_t^j}{\partial x}( x_t^{{\hat\psi^{\tau,j}(u)}}, \hat u_t^1,\dots, \hat u_t^N)\prod_{s=\tau+1}^{t-1}\frac{\partial f_s}{\partial x}(x_s^{{\hat\psi^{\tau,j}(u)}}, \hat u_s^1,\dots, \hat u_s^N),$$ 
		where $\hat\psi^{\tau,j}(u)=(\hat\psi^1,\dots,\hat \psi^{j-1},\hat\psi_j^\tau(u),\hat\psi^{j+1},\dots,\hat\psi^N)$ and  $\hat\psi_j^\tau(u)=(\hat u_0^j,\dots,\hat u_{\tau-1}^j,u,\hat u_{\tau+1}^j,\dots)$. 
		Given $\tau\in\nat$ and $j=1,\dots, N$, we suppose that there exists an open neighborhood $\mathcal O_\tau^j\subset U_\tau^j$ of $\hat u_\tau^j$ such that $\sum_{t=\tau+1}^{\infty}\rho^{\tau,j}_{t}$ converges uniformly on $\mathcal O_\tau^j$. 
		
	\end{assu}
	
	The next theorem follows from Theorem \ref{MP}.
	
	\begin{thm}\label{NEMP}
		Let $\hat \psi\in\Psi$ be such that Assumption \ref{AMPDG} holds. If $\hat\psi$ is a Nash equilibrium, then, for each $j=1,\dots, N$, there exists a sequence $\left\lbrace \lambda^j_t\right\rbrace_{t=1}^\infty$ in $\rzn$ such that
		\begin{enumerate}
			\item[(a)] For all $t\in\natu$,
			\begin{align}\label{MPXDG}
			\frac{\partial g^j_t}{\partial x}(x_t^{\hat\psi},\hat u_t^1,\dots,\hat u_t^N)+\lambda_{t+1}^j\frac{\partial f_t}{\partial x}(x_t^{\hat\psi},\hat u_t^1,\dots,\hat u_t^N)=\lambda^j_t,
			\end{align}
			\item[(b)] For all $t\in\nat$,
			\begin{align}\label{MPYDG}
			\frac{\partial g_t^j}{\partial y_j}(x_t^{\hat\psi},\hat u_t^1,\dots,\hat u_t^N)+\lambda^j_{t+1}\frac{\partial f_t}{\partial y_j}(x_t^{\hat\psi},\hat u_t^1,\dots,\hat u_t^N)=0,
			\end{align}
			\item[(c)] For all $h\in\natu$,
			\begin{align}\label{TCDG}
			\lim_{t\to\infty}\lambda^j_t \prod_{s=h}^{t-1}\frac{\partial f_s}{\partial x}(x_s^{\hat\psi},\hat u_s^1,\dots,\hat u_s^N)=0.
			\end{align}  
		\end{enumerate}
		Moreover, each $\lambda_t^j$ is given by 
		\begin{align}\label{DGmultipliers}
		\lambda^j_t=\sum_{k=t}^\infty\frac{\partial g^j_k}{\partial x}(x_k^{\hat\psi},\hat u_k^1,\dots,\hat u_k^N)\prod_{s=t}^{k-1}\frac{\partial f_s}{\partial x}(x_s^{\hat\psi},\hat u_s^1,\dots,\hat u_s^N).
		\end{align}
	\end{thm}

	\begin{assu}
		Let $\hat\psi\in\Psi$. We assume that the game (\ref{DGtuple}) satisfies the following for each $j=1,\dots,N$: 
		\begin{enumerate}\label{SuffConDG}
			\item[(a)]   $\Psi^j(x_0)$ is convex;
			
			\item[(b)]	
			the performance index $v^j$ is concave;
			
			\item[(c)] there exists a sequence of real  numbers $\left\lbrace m_t\right\rbrace $ such that $\sum_{t=0}^\infty m_t$ converges and 
			$
			g_t(x_t^{\hat\psi},\hat u_t^1,\dots, \hat u_{t}^{j-1},u_t^j,\hat u_t^{j+1},\dots, \hat u_t^N)\ge m_t
			$  for all $\psi^j=(u_0^j,u_1^j,\dots)\in\Psi^j(x_0)$.
		\end{enumerate}
	\end{assu}

	Theorem \ref{SufThm} yields the next theorem.
	
	\begin{thm}\label{DGSufThm}
		Let $\hat{\psi}\in\Psi$ be such that Assumption \ref{AMPDG} holds. Suppose that $\hat\psi$ satisfies  (\ref{MPXDG})-(\ref{TCDG}). If Assumption \ref{SuffConDG} holds, then $\hat\psi$ is a Nash Equilibrium.
	\end{thm}

	\begin{examp}
		Consider the following game with linear dynamics
		\begin{align*}
		x_{t+1}=x_t+u_t^1+\cdots+u_t^N,
		\end{align*}
		with $x_0\in\mathbb R$ given, and performance index 
		\begin{align*}
		\sum_{t=0}^\infty\beta^t\frac{1}{2}\left[x_t^2+[u_t^j]^2 \right] 
		\end{align*}
		for each player $j=1,\dots,N$.

		From (\ref{MPXDG})-(\ref{MPYDG}),
		for each $j=1,\dots,N$,
		\begin{align}
		\lambda_t^j=\beta^tx_t^{\hat{\psi}}+\lambda_{t+1}^j \hspace{0.3cm}\forall_{t\in\natu},
		\end{align}
		
		\begin{align}\label{2eqLQDG}
		0=\beta^t \hat u_t^j+\lambda^j_{t+1} \hspace{0.3cm}\forall_{t\in\nat}.
		\end{align}
		First, note that by (\ref{DGmultipliers}), $\lambda_t^j=\lambda_t^1$ for $j=1,\dots,N$. Moreover, by (\ref{2eqLQDG}), $u_t^j=u_t^1$ for $j=1,\dots,N$.
		Proceeding as in Example \ref{LQ}, we find $x_t^{\hat{\psi}}=x_0r^t$, where $r=\min\left\lbrace x|\beta x^2-[1+(1+N)\beta]x+1=0 \right\rbrace $. Thus,
		\begin{align*}
		\hat u_t^j&=\frac{x_{t+1}^{\hat\psi}-x_{t}^{\hat\psi}}{N}\\
		&=\frac{x_0(r-1)r^t}{N}\hspace{0.5cm}\forall_{j=1,\dots,N}
		\end{align*}
		Assumption \ref{AMPDG} can be proved exactly as Assumption \ref{AMP} was proved in Example \ref{LQ}.
	\end{examp}

	\begin{remark}
		Using the theory of Section \ref{Markovmodel}, we can prove an analogous maximum principle for dynamic games in which each player uses Markov strategies.
	\end{remark}

	
	
	
	\nocite{*}
	\bibliography{ref}

\begin{thebibliography}{14}
\expandafter\ifx\csname natexlab\endcsname\relax\def\natexlab#1{#1}\fi
\providecommand{\url}[1]{\texttt{#1}}
\providecommand{\href}[2]{#2}
\providecommand{\path}[1]{#1}
\providecommand{\DOIprefix}{doi:}
\providecommand{\ArXivprefix}{arXiv:}
\providecommand{\URLprefix}{URL: }
\providecommand{\Pubmedprefix}{pmid:}
\providecommand{\doi}[1]{\href{http://dx.doi.org/#1}{\path{#1}}}
\providecommand{\Pubmed}[1]{\href{pmid:#1}{\path{#1}}}
\providecommand{\bibinfo}[2]{#2}
\ifx\xfnm\relax \def\xfnm[#1]{\unskip,\space#1}\fi
\bibitem[{Acemoglu(2007)}]{EconomicGrowth}
\bibinfo{author}{Acemoglu, D.}, \bibinfo{year}{2007}.
\newblock \bibinfo{title}{Introduction to modern economic Growth}.
\newblock \bibinfo{type}{Levine's Bibliography}. UCLA Department of Economics.
\newblock \URLprefix
  \url{https://EconPapers.repec.org/RePEc:cla:levrem:122247000000001721}.
\bibitem[{Aseev et~al.(2017)Aseev, Krastanov and Veliov}]{Multipliers}
\bibinfo{author}{Aseev, S.M.}, \bibinfo{author}{Krastanov, M.I.},
  \bibinfo{author}{Veliov, V.M.}, \bibinfo{year}{2017}.
\newblock \bibinfo{title}{Optimality conditions for discrete-time optimal
  control on infinite horizon}.
\newblock \bibinfo{journal}{Pure Appl. Funct. Anal.} \bibinfo{volume}{2},
  \bibinfo{pages}{395--409}.
\bibitem[{{Bachir} and {Fabre}(2018)}]{Bachir}
\bibinfo{author}{{Bachir}, M.}, \bibinfo{author}{{Fabre}, A.},
  \bibinfo{year}{2018}.
\newblock \bibinfo{title}{{G\^ateaux-differentiability of convex functions in
  infinite dimension}}.
\newblock \bibinfo{journal}{ArXiv e-prints}
  \href{http://arxiv.org/abs/1802.07633}{{\tt arXiv:1802.07633}}.
\bibitem[{Blot and Hayek(2014)}]{Blot}
\bibinfo{author}{Blot, J.}, \bibinfo{author}{Hayek, N.l.},
  \bibinfo{year}{2014}.
\newblock \bibinfo{title}{Infinite-horizon optimal control in the discrete-time
  framework}.
\newblock SpringerBriefs in Optimization, \bibinfo{publisher}{Springer, New
  York}.
\newblock \URLprefix \url{https://doi.org/10.1007/978-1-4614-9038-8},
  \DOIprefix\doi{10.1007/978-1-4614-9038-8}.
\bibitem[{Bourdin and Tr\'elat(2013)}]{Bourdin}
\bibinfo{author}{Bourdin, L.c.}, \bibinfo{author}{Tr\'elat, E.},
  \bibinfo{year}{2013}.
\newblock \bibinfo{title}{Pontryagin maximum principle for finite dimensional
  nonlinear optimal control problems on time scales}.
\newblock \bibinfo{journal}{SIAM J. Control Optim.} \bibinfo{volume}{51},
  \bibinfo{pages}{3781--3813}.
\newblock \URLprefix \url{https://doi.org/10.1137/130912219},
  \DOIprefix\doi{10.1137/130912219}.
\bibitem[{C.~Chow(1997)}]{LagrangeMethod}
\bibinfo{author}{C.~Chow, G.}, \bibinfo{year}{1997}.
\newblock \bibinfo{title}{Dynamic Economics: Optimization by the Lagrange
  method}.
\newblock \bibinfo{publisher}{New York : Oxford University Press}.
\bibitem[{Clarke(2013)}]{Clarke}
\bibinfo{author}{Clarke, F.}, \bibinfo{year}{2013}.
\newblock \bibinfo{title}{Functional analysis, calculus of variations and
  optimal control}. volume \bibinfo{volume}{264} of
  \textit{\bibinfo{series}{Graduate Texts in Mathematics}}.
\newblock \bibinfo{publisher}{Springer, London}.
\newblock \URLprefix \url{https://doi.org/10.1007/978-1-4471-4820-3},
  \DOIprefix\doi{10.1007/978-1-4471-4820-3}.
\bibitem[{Fleming and Rishel(1975)}]{Flemming}
\bibinfo{author}{Fleming, W.H.}, \bibinfo{author}{Rishel, R.W.},
  \bibinfo{year}{1975}.
\newblock \bibinfo{title}{Deterministic and stochastic optimal control}.
\newblock \bibinfo{publisher}{Springer-Verlag, Berlin-New York}.
\newblock \bibinfo{note}{Applications of Mathematics, No. 1}.
\bibitem[{Gamkrelidze(2006)}]{HistoryMP}
\bibinfo{author}{Gamkrelidze, R.V.}, \bibinfo{year}{2006}.
\newblock \bibinfo{title}{Discovery of the maximum principle}, in:
  \bibinfo{booktitle}{Mathematical events of the twentieth century}.
  \bibinfo{publisher}{Springer, Berlin}, pp. \bibinfo{pages}{85--99}.
\newblock \URLprefix \url{https://doi.org/10.1007/3-540-29462-75},
  \DOIprefix\doi{10.1007/3-540-29462-75}.
\bibitem[{Gonz\'alez-S\'anchez and Hern\'andez-Lerma(2013)}]{DavidEE}
\bibinfo{author}{Gonz\'alez-S\'anchez, D.}, \bibinfo{author}{Hern\'andez-Lerma,
  O.}, \bibinfo{year}{2013}.
\newblock \bibinfo{title}{Discrete-time stochastic control and dynamic
  potential games: The Euler-equation approach}.
\newblock SpringerBriefs in Mathematics, \bibinfo{publisher}{Springer, New
  York}.
\newblock \URLprefix \url{https://doi.org/10.1007/978-3-319-01059-5},
  \DOIprefix\doi{10.1007/978-3-319-01059-5}.
\bibitem[{Gonz\'alez-S\'anchez and Hern\'andez-Lerma(2014)}]{EEDavid}
\bibinfo{author}{Gonz\'alez-S\'anchez, D.}, \bibinfo{author}{Hern\'andez-Lerma,
  O.}, \bibinfo{year}{2014}.
\newblock \bibinfo{title}{On the {E}uler equation approach to discrete-time
  nonstationary optimal control problems}.
\newblock \bibinfo{journal}{J. Dyn. Games} \bibinfo{volume}{1},
  \bibinfo{pages}{57--78}.
\newblock \URLprefix \url{https://doi.org/10.3934/jdg.2014.1.57},
  \DOIprefix\doi{10.3934/jdg.2014.1.57}.
\bibitem[{Ljungqvist and Sargent(2004)}]{Ljungq}
\bibinfo{author}{Ljungqvist, L.}, \bibinfo{author}{Sargent, T.},
  \bibinfo{year}{2004}.
\newblock \bibinfo{title}{Recursive macroeconomic theory, 2nd Edition}.
  volume~\bibinfo{volume}{1}.
\newblock \bibinfo{edition}{2} ed., \bibinfo{publisher}{The MIT Press}.
\newblock \URLprefix
  \url{https://EconPapers.repec.org/RePEc:mtp:titles:026212274x}.
\bibitem[{Luenberger(1969)}]{Lue}
\bibinfo{author}{Luenberger, D.}, \bibinfo{year}{1969}.
\newblock \bibinfo{title}{{Optimization by vector space methods}}.
\newblock \bibinfo{publisher}{Wiley-Interscience}.
\bibitem[{Pontryagin et~al.(1962)Pontryagin, Boltyanskii, Gamkrelidze and
  Mishchenko}]{Pont}
\bibinfo{author}{Pontryagin, L.S.}, \bibinfo{author}{Boltyanskii, V.G.},
  \bibinfo{author}{Gamkrelidze, R.V.}, \bibinfo{author}{Mishchenko, E.F.},
  \bibinfo{year}{1962}.
\newblock \bibinfo{title}{The mathematical theory of optimal processes}.
\newblock Translated from the Russian by K. N. Trirogoff; edited by L. W.
  Neustadt, \bibinfo{publisher}{Interscience Publishers John Wiley \& Sons,
  Inc.\, New York-London}.

\end{thebibliography}
	\bibliographystyle{elsarticle-har}

\end{document}